\documentclass[a4paper, 11pt]{article}

\usepackage[dvipsnames]{xcolor}
\usepackage[pdfencoding=unicode, hidelinks]{hyperref}
\hypersetup{
	colorlinks,
	linkcolor={red},
	citecolor={blue},
}
\usepackage[english]{babel}
\usepackage[utf8]{inputenc}
\usepackage{amsthm}
\usepackage{amsmath}
\usepackage{amssymb}
\usepackage{mathrsfs}
\usepackage[mathscr]{eucal}
\usepackage{wasysym}
\usepackage{textcomp}
\usepackage{enumerate}
\usepackage{setspace}
\usepackage{latexsym}
\usepackage{graphicx}
\usepackage{mathtools}
\usepackage{tabularray}
\usepackage{geometry}
\usepackage{float}
\usepackage{emptypage}
\usepackage{commath}
\usepackage{multicol}
\usepackage{titlesec}
\usepackage[shortlabels]{enumitem}
\usepackage{verbatim}
\usepackage{cases} 
\usepackage[noabbrev]{cleveref}
\usepackage{mathrsfs}
\usepackage{authblk}
\usepackage{orcidlink}

\usepackage{csquotes}
\usepackage[backend=biber,style=alphabetic,citestyle=alphabetic, doi=false,url=false,isbn=false, maxnames=999, sorting=nyt]{biblatex}
\usepackage{hyperref}
\theoremstyle{plain}
\newtheorem{thm}{Theorem}[section]
\theoremstyle{plain}

\theoremstyle{plain}
\newtheorem{lemma}[thm]{Lemma}
\theoremstyle{plain}
\newtheorem{defn}[thm]{Definition}
\theoremstyle{plain}
\newtheorem{prop}[thm]{Proposition}
\theoremstyle{definition}
\newtheorem{remark}[thm]{Remark}

\numberwithin{equation}{section}

\DeclareMathOperator{\interior}{int}
\newcommand{\RR}{\mathbb{R}}
\newcommand{\Rd}{\mathbb{R}^d}
\newcommand{\NN}{\mathbb{N}}

\renewcommand{\norm}[1]{\lVert#1\rVert}
\newcommand{\duality}[2]{\langle#1,#2\rangle}

\DeclareMathOperator{\diver}{div}
\newcommand{\dx}{\,\ensuremath{\mathrm{d}}x}
\newcommand{\ds}{\,\ensuremath{\mathrm{d}}s}
\newcommand{\dt}{\,\ensuremath{\mathrm{d}}t}
\newcommand{\dA}{\,\ensuremath{\mathrm{d}}\mathcal{H}^{d-1}}
\newcommand{\nnu}{\boldsymbol{\nu}}

\newcommand{\R}{\mathcal{R}}
\newcommand{\C}{\mathcal{C}}
\newcommand{\T}{\mathcal{T}}
\newcommand{\V}{\mathcal{V}}
\newcommand{\uu}{\boldsymbol{u}}
\newcommand{\epsilonu}{\varepsilon(\boldsymbol{u})}
\newcommand{\epsilonut}{\varepsilon(\partial_t\uu)}
\newcommand{\vv}{\boldsymbol{v}}
\newcommand{\epsilonv}{\varepsilon(\boldsymbol{v})}
\newcommand{\Lambdac}{\Lambda_c}

\makeatletter
\newcommand{\oset}[2]{%
  {\operatorname*{#2}\limits^{\vbox to \ex@{\kern-\tw@\ex@
   \hbox{\scriptsize #1}\vss}}}}
\makeatother

\newcommand{\psiu}{\oset{$\smallsmile$}{\Psi}{}}
\newcommand{\psin}{\oset{$\smallfrown$}{\Psi}{}}
\newcommand{\hu}{\oset{$\smallsmile$}{h}{}}
\newcommand{\hn}{\oset{$\smallfrown$}{h}{}}
\newcommand{\Wzu}{\oset{$\smallsmile$}{W}_{3,z}{}}
\newcommand{\Wzn}{\oset{$\smallfrown$}{W}_{3,z}{}}

\newcommand{\into}{\int_{\Omega}}
\newcommand{\intto}{\int_0^T \!\!\!\into}
\newcommand{\intt}{\int_0^T}
\newcommand{\intg}{\int_{\Gamma}}
\newcommand{\intTg}{\int_0^T \!\!\!\intg}
\newcommand{\intk}{\int_{t_{\tau}^{k-1}}^{t_{\tau}^k}}
\newcommand{\Dk}{D_{\tau,k}}

\newcommand{\fk}{f_{\tau}^k}
\newcommand{\phik}{\varphi_{\tau}^k}
\newcommand{\muk}{\mu_{\tau}^k}
\newcommand{\sigmak}{\sigma_{\tau}^k}
\newcommand{\sigmack}{\sigma_{c,\tau}^k}
\newcommand{\sigmagk}{\sigma_{\Gamma,\tau}^k}
\newcommand{\uk}{\uu_{\tau}^k}
\newcommand{\epsilonuk}{\varepsilon(\uu_{\tau}^k)}
\newcommand{\epsilonudk}{\varepsilon(\Dk\uu)}
\newcommand{\vk}{\vv_{\tau}^k}
\newcommand{\epsilonvk}{\varepsilon(\vv_{\tau}^k)}
\newcommand{\zk}{z_{\tau}^k}

\newcommand{\phikm}{\varphi_{\tau}^{k-1}}
\newcommand{\mukm}{\mu_{\tau}^{k-1}}
\newcommand{\sigmakm}{\sigma_{\tau}^{k-1}}

\newcommand{\ukm}{\uu_{\tau}^{k-1}}
\newcommand{\ukmm}{\uu_{\tau}^{k-2}}
\newcommand{\epsilonukm}{\varepsilon(\uu_{\tau}^{k-1})}

\newcommand{\vkm}{\vv_{\tau}^{k-1}}
\newcommand{\epsilonvkm}{\varepsilon(\vv_{\tau}^{k-1})}
\newcommand{\zkm}{z_{\tau}^{k-1}}
\newcommand{\hkm}{h(\zkm)}

\newcommand{\vl}{\vv_{\tau}^l}

\newcommand{\phij}{\varphi_{\tau}^j}

\newcommand{\muj}{\mu_{\tau}^j}
\newcommand{\sigmaj}{\sigma_{\tau}^j}
\newcommand{\epsilonuj}{\varepsilon(\uu_{\tau}^j)}
\newcommand{\vj}{\vv_{\tau}^j}
\newcommand{\zj}{z_{\tau}^j}

\newcommand{\phit}{\varphi_{\tau}}
\newcommand{\phito}{\overline{\varphi}_{\tau}}
\newcommand{\phitu}{\underline{\varphi}_{\tau}}
\newcommand{\muto}{\overline{\mu}_{\tau}}
\newcommand{\mutu}{\underline{\mu}_{\tau}}
\newcommand{\sigmat}{\sigma_{\tau}}
\newcommand{\sigmato}{\overline{\sigma}_{\tau}}
\newcommand{\sigmatu}{\underline{\sigma}_{\tau}}
\newcommand{\sigmacto}{\overline{\sigma}_{c,\tau}}
\newcommand{\sigmagto}{\overline{\sigma}_{\Gamma,\tau}}
\newcommand{\uto}{\overline{\uu}_{\tau}}
\newcommand{\utu}{\underline{\uu}_{\tau}}
\newcommand{\ut}{\uu_{\tau}}
\newcommand{\epsilonuto}{\varepsilon(\uto)}
\newcommand{\epsilonutu}{\varepsilon(\utu)}
\newcommand{\vto}{\overline{\vv}_{\tau}}
\newcommand{\vtu}{\underline{\vv}_{\tau}}
\newcommand{\vt}{\vv_{\tau}}
\newcommand{\epsilonvto}{\varepsilon(\vto)}

\newcommand{\zto}{\overline{z}_{\tau}}
\newcommand{\ztu}{\underline{z}_{\tau}}
\newcommand{\zt}{z_{\tau}}

\newcommand{\fto}{\overline{f}_{\tau}}

\newcommand{\phid}{\varphi_{\delta}}

\crefname{lemma}{lemma}{lemmas}
\Crefname{lemma}{Lemma}{Lemmas}
\crefname{prop}{proposition}{proposition}
\Crefname{prop}{Proposition}{Propositions}
\crefname{cor}{corollary}{corollaries}
\Crefname{cor}{Corollary}{Corollaries}
\crefname{remark}{remark}{remarks}
\Crefname{remark}{Remark}{Remarks}
\crefname{thm}{theorem}{theorems}
\Crefname{thm}{Theorem}{Theorems}
\Crefname{section}{Section}{Sections}

\title{A phase field model of Cahn--Hilliard type for tumour growth with mechanical effects and damage}

\author{\large{\textsc{Giulia Cavalleri}} \orcidlink{0009-0006-4154-9659}}

\affil{\normalsize{Dipartimento di matematica “F. Casorati”, Università degli Studi di Pavia,\\
via Ferrata 5, 27100 Pavia, Italy\\
E-mail: \texttt{giulia.cavalleri01@universitadipavia.it}}}

\date{}

\begin{document}

\maketitle
\let\thefootnote\relax\footnotetext{This is a post-peer-review, pre-copyedit version of an article published in \textit{Journal of Mathematical Analysis and Applications}. The final authenticated version is available in Open Access at \url{https://doi.org/10.1016/j.jmaa.2025.129627}.}

\begin{abstract}
     \noindent We introduce a new diffuse interface model for tumour growth in the presence of a nutrient, in which we take into account mechanical effects and reversible tissue damage. The highly nonlinear PDEs system mainly consists of a Cahn--Hilliard type equation that describes the phase separation process between healthy and tumour tissue coupled to a parabolic reaction-diffusion equation for the nutrient and a hyperbolic equation for the balance of forces, including inertial and viscous effects. The main novelty of this work is the introduction of cellular damage, whose evolution is ruled by a parabolic differential inclusion. In this paper, we prove a global-in-time existence result for weak solutions by passing to the limit in a time-discretised and regularised version of the system.\\
     
    
\end{abstract}

\section{Introduction}

Cancer is one of the leading causes of death worldwide, and understanding the primary mechanisms underlying its development is one of the main challenges scientists face nowadays. 
Genetic, biochemical, and mechanical processes come into play simultaneously, making it difficult to predict the course of the disease and design specific and effective treatments.
For this reason, it has been understood that mathematics can be fundamental, offering quantitative tools that can significantly enhance diagnostic and prognostic applications. Over the last decades, there has been an increasing interest in mathematical modelling for tumour growth, see, e.g., \cite{Byrne_etal_2006, Cristini_Lowengrub_2010,Asantin_Preziosi_2008} and the references cited therein. In particular, among the various possible modelling approaches, we will focus on the so-called phase-field or diffuse interface models. At first glance, it might seem intuitive to model solid tumours as masses separated from healthy tissue by a sharp interface, employing a free boundary problem (see \cite{Byrne_etal_97,Friedman_07}). However, these models present technical limitations in describing situations where there is a topological change in the tumour, such as coalescence or breaking-up phenomena, which typically occur both at the early stages of the proliferation (when the tumour is morphologically unstable, see \cite{Cristini_etal_03}) and at more advanced stages (when it undergoes metastasis). This difficulty can be overcome by employing a diffuse interface model, in which the sharp interface is replaced by a thin transition layer with both tumour and healthy cells. Without attempting to be exhaustive, we refer to \cite{Asantin_Preziosi_2008, Byrne_1999, Lowengrub_Frieboes_Jin_,  Colli_2015, Frigeri_Grasselli_Rocca, Garcke_Lam_Sitka, Miranville_Rocca_Schimperna} and the references cited therein. This work aims to introduce and study a new mathematical phase-field model for the evolution of a young tumour, which implies that the tumour is in the avascular phase and there is no differentiation between different types of tumour cells (viable, quiescent, and necrotic). As it is common, the tumour growth process is ruled by a Cahn-Hilliard type equation (see, e.g., \cite{Miranville_book} or \cite{Hao_review} for further details on the classical Cahn-Hilliard equation) coupled with other equations describing the behaviour of relevant quantities. We will take into account the following aspects.
\begin{enumerate}[(i)]
    \item \label{enumerate:nutrient} The presence of the \textit{nutrient}, a chemical species that feeds tumour cells (such as oxygen or glucose). In our setting, it is provided by the pre-existing vasculature, since we assume that the tumour has not developed its own yet. 
    \item  \label{enumerate:viscoelasticity} The \textit{viscoelastic} behaviour of biological tissues, which exhibit both elastic (instantaneous response to stress) and viscous (time-dependent deformation) properties. Moreover, it is well known that solid stress can affect tumour growth (see, e.g., \cite{Urcun_Lorenzo_etal_22}) and, at the same time, tumour growth increases mechanical stress. We assume infinitesimal displacements, so we will work in the case of linear elasticity.
    \item  \label{enumerate:damage} The local \textit{tissue damage} caused by surgery.  In many cases, the standard of care requires surgical resection of the tumour: this causes lesions that, in turn, affect the proliferation of tumour cells when the growth process eventually restarts. This may happen for several reasons. First, removing part of the tissue causes damage to the blood vessels and edema: this must be taken into account in the nutrient equation. Second, the surgical groove is characterised by different elastic properties compared to intact tissue (see, e.g., \cite{Moeendarbary2017}), which must be considered when choosing a suitable form for the elastic energy.
\end{enumerate}
While the influence of \ref{enumerate:nutrient} and \ref{enumerate:viscoelasticity} on tumour growth is already deeply investigated in the literature (see, e.g., \cite{Garcke_Lam_Signori_21,Garcke_Kovacs_Trautwein_22,Garcke_Trautwein_2024}), the role of \ref{enumerate:damage} is a complete novelty in this field. However, the impact of the damage and (visco-)elasticity in phase separation processes has been thoroughly explored in various modelling studies within the field of materials science (see, e.g., \cite{heinemann_kraus_2011,heinemann_kraus_2015, heinemann_kraus_rocca_rossi_2017}).\\

\noindent \textbf{The PDEs system.} Explicitly, we derive the following PDEs system
\begin{subequations}\label{eq:problem}
    \begin{align}
        & \partial_t\varphi - \Delta \mu = U(\varphi, \sigma, \varepsilon(\uu), z), \label{eq:ch1}\\
        & \mu = - \epsilon \Delta \varphi + \frac{1}{\epsilon}\Psi'(\varphi) + W_{,\varphi}(\varphi, \epsilonu, z), \label{eq:ch2}\\
        & \partial_t\sigma - \Delta \sigma = S(\varphi, \sigma, z), \label{eq:nutrient}\\
        & \kappa \partial_{tt}\uu - \diver\left[a(z)\V \varepsilon(\partial_t\uu) + W_{,\varepsilon}(\varphi, \epsilonu, z)   \right] = \mathbf{0}, \label{eq:displacement}\\
        & \partial_t z - \Delta_p z + \beta(z) + \pi(z)  + W_{,z}(\varphi, \epsilonu, z) \ni 0, \label{eq:damage}
    \end{align}
\end{subequations}
posed in $Q \coloneqq \Omega \times (0,T)$, where $\Omega$ is a smooth enough domain in $\RR^d$ with $d=2,3$ and $T>0$ is a fixed time. The Cahn--Hilliard equation given by the combination of \eqref{eq:ch1}--\eqref{eq:ch2} describes the phase separation process between healthy and tumour tissue, where $\varphi$ denotes the local difference in volume fraction between tumour and healthy cells. This means that, at least in principle, the set $\{\varphi=-1\}$ corresponds the healthy tissue, $\{\varphi=1\}$ to the tumour tissue and $\{-1<\varphi<1\}$ is the diffuse interface that separates them (see \Cref{remark:phi_physical_interval}).  The parameter $\epsilon$ represents the thickness of the interfacial layer.  The chemical potential associated to $\varphi$ is denoted by $\mu$. The reaction-diffusion equation \eqref{eq:nutrient} rules the diffusion of $\sigma$, that is, the concentration of the nutrient. The hyperbolic equation \eqref{eq:displacement} describes the dynamics for $\uu$, the small displacement field of each point with respect to the reference undeformed configuration. Here, $\varepsilon(\uu)$ is the symmetric gradient of $\uu$, i.e., $\varepsilon(\uu) = \frac{1}{2}\left(\nabla\uu + (\nabla \uu)^t\right)$. The fixed and positive parameter $\kappa$ is supposed to be small and represents the fact that tumour growth occurs at a much larger timescale than the tissue relaxation into mechanical equilibrium. For simplicity and without any loss of generality, later on, we will set $\epsilon = \kappa = 1$. Finally,  the differential inclusion \eqref{eq:damage} represents the evolution law for local tissue damage $z$, which is the main novelty we introduce. Classically, the damage takes values between $0$ and $1$: if $z(x)$ is equal to $1$, there is no damage at the point $x\in \Omega$, $z(x)$ equal to $0$ means that there is complete damage and an intermediate value indicates partial damage (see \Cref{remark:z_physical_interval}).\\

\subsection{Derivation of the model}\label{subsect:derivation_of_the_model}
The evolution of our system is driven by classical thermodynamic principles and relies on a total energy $\mathscr{E}$ and a pseudopotential of dissipation $\mathscr{P}$. The total energy of our system is
\begin{equation*}
    \mathscr{E}(\varphi, \sigma, \uu, z)=\into E(\varphi,\nabla \varphi, \sigma,\varepsilon(\uu), \partial_t \uu, z, \nabla z) \dx
\end{equation*}
where the energy density $E$ is the sum of a generalized free energy density and the kinetic energy density. We postulate it has the following form:
\begin{equation*}
   \begin{split}
       E(&\varphi,\nabla \varphi, \sigma,\varepsilon(\uu), \partial_t \uu, z, \nabla z) \\ &= \frac{1}{2}|\nabla \varphi|^2 + \Psi(\varphi)
       + \frac{1}{2} |\sigma|^2  + \frac{1}{2}|\partial_t \uu|^2 +\frac{1}{p} |\nabla z|^p + \widehat{\beta}(z)
       +\widehat{\pi}(z)  + W(\varphi,\varepsilon(\uu),z).
   \end{split}
\end{equation*}
The term $\frac{1}{2}|\nabla \varphi|^2 + \Psi(\varphi)$ is the classical contribution of Ginzburg--Landau type and accounts for the interfacial energy of the diffuse interface. The addend $\frac{1}{2}|\sigma|^2$  results from the presence of the nutrient, in the sense that higher nutrient concentration corresponds to higher energy. The system's kinetic energy is given by $\frac{1}{2}|\partial_t \uu|^2$.  Regarding the damage, $\frac{1}{p} |\nabla z|^p + \widehat{\beta}(z) +\widehat{\pi}(z)$ is an interaction free energy. According to the gradient theories in damage processes, the gradient term models the influence of damage at a material point, undamaged in the surrounding. The non-smooth convex $\widehat{\beta}$ allows us to impose physical constraints on the variable $z$ (such as requiring that $z \in [0,1]$), while $\widehat{\pi}$ is a smooth perturbation with at most quadratic growth. Lastly,  as already anticipated, $W$ is the elastic energy density. To include dissipation in our model, we define a pseudo-potential of dissipation 
\begin{equation*}
    \mathscr{P}(\varepsilon(\partial_t \uu), \partial_t z) = \into P(\varepsilon(\partial_t \uu), \partial_t z) \dx,
\end{equation*}
where 
\begin{equation*}
    P(\varepsilon(\partial_t \uu), \partial_t z) = \frac{1}{2} a(z)\epsilonut : \V \epsilonut +  \frac{1}{2}|\partial_t z|^2.
\end{equation*}
It depends on the damage time derivative $\partial_t z$ and the macroscopic symmetric strain rate $\epsilonut$, which are the dissipative variables of our problem. The fourth-order viscous tensor term $\V$ represents the friction between adjacent moving cells with different velocities.
Notice that $P$ depends also on the damage $z$,  but we use the notation $P$ instead of a more precise $P_z$ for brevity. 
\noindent Following Gurtin's approach \cite{Gurtin_1996}, our system can be derived starting from balance laws for the involved quantities and then imposing constitutive assumptions so that the system satisfies the second law of thermodynamics, which, in the case of an isothermal system like ours, is written in the form of an energy dissipation inequality (see, e.g., \cite{Garcke_Lam_Sitka, heinemann_kraus_rocca_rossi_2017}).\\
The Cahn--Hilliard equation of the system \eqref{eq:ch1}--\eqref{eq:ch2} is derived from the mass balance law
\begin{equation*}
    \partial_t \varphi + \diver \mathbf{J_\varphi} = U,
\end{equation*}
where $\mathbf{J_\varphi}$ is the mass flux and $U$ is a mass source. As usual, the mass flux is prescribed by the following constitutive equation
\begin{equation*}
    \mathbf{J_{\varphi}} \coloneqq - \nabla \mu,
\end{equation*}
where $\mu$ is the chemical potential associated with $\varphi$ and it is defined as the variational derivative of the energy with respect to $\varphi$, i.e., 
\begin{equation*}
    \mu \coloneqq \frac{\delta \mathscr{E}}{\delta \varphi}= -\Delta \varphi + \Psi'(\varphi)+ W_{,\varphi}(\varphi,\epsilonu,z).
\end{equation*}
Here we adopt the standard notation according to which $W_{,\varphi}$ is the derivative of $W$ with respect to $\varphi$ and the same for the other variables.\\
The \cref{eq:nutrient} for the evolution of the nutrient is also derived from a mass balance law,
\begin{equation*}
    \partial_t \sigma + \diver \mathbf{J_{\sigma}} = S,
\end{equation*}
where $S$ is a source/sink of nutrients and the mass flux  is chosen as
\begin{equation*}
    \mathbf{J_{\sigma}} \coloneqq -\nabla \left[\partial_\sigma E\right] =- \nabla \sigma.
\end{equation*}
The equation \eqref{eq:displacement} governing the displacement is a balance law for macroscopic movements in which inertial effects are taken into account and external forces are neglected, derived from the principle of virtual power
\begin{equation*}
    \partial_{tt}\uu - \diver \T = 0.
\end{equation*}
Here $\T$ is the stress tensor and we postulate it is the sum of a non-dissipative (elastic stress) and a dissipative part (viscous stress) given by
\begin{equation*}
    \T = \T^{\text{nd}} + \T^{\text{d}}= \partial_{\epsilonu} E+ \partial_{\epsilonut} P = W_{,\varepsilon}(\varphi, \epsilonu, z)+a(z)\V \varepsilon(\partial_t\uu).
\end{equation*}
Finally, the damage differential inclusion \eqref{eq:damage} is derived from the micro-force balance law
\begin{equation*}
    B - \diver  \mathbf{H} = 0,
\end{equation*}
where we assume that the sum of the external micro-forces acting on the body is equal to zero. The quantity $B$ represents the internal micro-forces and is defined by the following  constitutive assumption
\begin{align*}
    B=B^{\text{nd}}+B^{\text{d}}\qquad 
    \text{with} \qquad
    &B^{\text{nd}} \in \partial_{z}E= \partial \widehat{\beta}(z)+ \widehat{\pi}'(z)+W_{,z}(\varphi,\epsilonu,z),\\ 
     &B^{\text{d}} = \partial_{\partial_t z}P=\partial_t z.
     \end{align*}
Without entering into the mathematical details, $\partial_{z}E$ and $\partial \widehat{\beta}$ have to be interpreted as subdifferentials in the sense of convex analysis and this justifies the presence of the belonging symbol instead of equality. In the following, we will employ the notation $\beta \coloneqq \partial \widehat{\beta}$ and $\pi \coloneqq \widehat{\pi}'$. The term $\mathbf{H}$ is the internal micro-stress and is defined by
 \begin{align*}
    \mathbf{H} = \mathbf{H^{\text{nd}}} + \mathbf{H^{\text{d}}} \qquad 
    \text{with}\qquad &\mathbf{H}^{\text{nd}}=\partial_{\nabla z}E=|\nabla z|^{p-2}\nabla z,\\
    &\mathbf{H}^{\text{d}}=\partial_{\nabla (\partial_t z)}P=0.
\end{align*}
Notice that here we followed Frémond's approach (see \cite{Fremond_2002}), assuming that $\T$, $B$ and $\mathbf{H}$ can be additively decomposed in a dissipative and a non-dissipative part.\\ 

\noindent \textbf{Boundary and initial conditions.} We assume the system is isolated from the exterior, so we prescribe no-flux conditions for $\varphi$, $\mu$, and $z$. Regarding $\sigma$, we allow a more general Robin condition that may model also the boundary supply of the nutrient. We assume that $\uu$ is zero at the boundary, as in the situation in which the domain is delimited by a rigid part of the body (e.g., a bone) that prevents displacements. Namely, we couple the previous system \eqref{eq:problem} with the following boundary conditions
\begin{subequations}\label{eq:boundary_conditions}
    \begin{align}
        & \nabla\varphi \cdot \nnu = \nabla \mu \cdot \nnu = 0,\label{eq:phi,mu_nu}\\
        & \nabla\sigma \cdot \nnu + \alpha(\sigma-\sigma_{\Gamma}) = 0,\label{eq:sigma_nu}\\
        & \uu = \mathbf{0},\label{eq:u_nu}\\
        & (|\nabla z|^{p-2}\nabla z) \cdot \nnu = 0,\label{eq:z_nu}
    \end{align}
\end{subequations}
 on $\Sigma \coloneqq \Gamma \times (0,T)$, where $\Gamma \coloneqq \partial \Omega$ is the boundary of the domain and $\nnu$ is the outward unit normal to $\Gamma$. The term $\sigma_{\Gamma}$ is the prescribed concentration of the nutrient at the boundary, and $\alpha$ is a given non-negative constant. Notice that, if $\alpha = 0$, we gain a no-flux condition also for the nutrient. 
 The system is supplemented with the initial conditions 
\begin{equation}\label{eq:initial_conditions}
    \varphi(0) = \varphi_0,\quad
    \sigma(0) = \sigma_0, \quad
    \uu(0) = \uu_0, \quad
    \partial_t\uu(0)=\vv_0, \quad
    z(0) = z_0,
\end{equation}
in $\Omega$.\\

\noindent \textbf{Choice of the sources.} The nonlinear source $U$ in equation \eqref{eq:ch1} accounts for biological mechanisms related to tumour cells proliferation and death. Explicitly, we make the following choice
\begin{equation}\label{eq:defn_U}
   U(\varphi, \sigma, \varepsilon(\uu), z) \coloneqq \bigg(\frac{\lambda_p \sigma}{1+ |W_{,\varepsilon}(\varphi,\varepsilon(\uu),z)|} - \lambda_a + f\bigg)g(\varphi, z),
\end{equation}
referring to \cite{Garcke_Lam_2017,Garcke_Lam_Signori_21}.
As it is common, we assume the mechanisms controlling cell division to be suppressed in tumour cells, so proliferation is limited only by the availability of nutrients. We model it with the term $\lambda_p \sigma$, where $\lambda_p$ is a fixed proliferation coefficient. We also suppose that tumour cells only die because of apoptosis, and we denote with $\lambda_a$ the constant apoptosis rate. Furthermore, we consider the presence of mechanical stress caused by surrounding tissues as a factor that can reduce tumour growth. This is expressed by the fact that, if the mechanical stress $W_{,\varepsilon}$ grows in modulus, the proliferation term $\lambda_p \sigma$ reduces. We also allow the presence of a medical treatment, modelled by the prescribed function $f$, that affects proliferation. The function $g$ guarantees that proliferation and apoptosis occur only in the tumour tissue, as well as the effectiveness of the medical care $f$. A good modelling choice is a non-negative function that vanishes in $\{\varphi = -1\}$, is equal to 1 where $\{\varphi = 1\}$ and is increasing in the variable $\varphi$. We also allow the dependence of $g$ on the damage $z$.\\
For the choice of the nutrient source $S$ in \cref{eq:nutrient}, we refer to the aforementioned literature, assuming
\begin{equation}\label{eq:defn_S}
    S(\varphi, \sigma, z) \coloneqq -\lambda_c\sigma g(\varphi,z) + \Lambdac(z) (\sigma_c - \sigma). 
\end{equation}
The term $-\lambda_c\sigma g(\varphi,z)$ models the fact that the nutrient consumption is higher where the tumour cells density is higher. Here, $\lambda_c$ is a fixed consumption rate. The term $\Lambdac(z)(\sigma_c -\sigma)$ is a supply term that takes into account the nutrients provided by the nearby capillaries. Note that the capillary supply rate $\Lambdac$ may depend on the local damage $z$ since the damage, in the sense of a lesion caused by a surgical procedure, affects the blood vessels.\\

\noindent \textbf{Choice of the elastic energy density.} Accordingly to the classical theory of linear elasticity (see, e.g., \cite{Slaughter_book_2002}) and to the previous literature (see, e.g., \cite{Garcke_Lam_Signori_21, Garcke_Kovacs_Trautwein_22}), we assume that the elastic energy density has the following expression
\begin{equation}\label{eq:elastic_energy_expression}
    W(\varphi, \epsilonu, z) \coloneqq W(x,\varphi, \epsilonu, z) = \frac{1}{2}h(z)\C(x)(\epsilonu-\R\varphi): (\epsilonu-\R \varphi).
\end{equation}
Notice that, even if $W$ may depend on the space variable $x$, with a slight abuse of notation, we will omit this dependence in the following. Here, $\mathcal{C}$ is a fourth-order elasticity tensor whose mathematical requirements will be specified later on. We include the multiplicative, non-negative, and possibly degenerate function $h$ to add dependence on the damage. Notice that, from the modelling point of view, $\C$ should also depend on the phase $\varphi$ because tumour tissue and healthy tissue could have a different elastic response to solicitations. However, we weren't able to handle such dependence from the mathematical point of view (see \Cref{subsect:derivation_of_the_model} below). Finally, the term $\R \varphi$ is the stress-free strain (also called eigenstrain), which is the strain the material would attain if the tissue were uniform and unstressed at a phase configuration $\varphi$. In other words, it is the strain due to growth. As it is common, we assume that it satisfies Vegard's law, i.e., it is given by a linear function of $\varphi$, where $\R \in \RR^{d \times d}$ is a fixed matrix. With such a choice, the partial derivatives of $W$ that appear in the equations of the PDEs system are:
\begin{alignat}{2}
    &W_{,\varphi}(\varphi,\epsilonu,z) &&= -h(z)\C(\epsilonu-\R \varphi) : \R, \label{eq:W_deriv_phi}\\ 
    &W_{,\varepsilon}(\varphi,\epsilonu,z) &&=h(z)\C(\epsilonu-\R \varphi),\label{eq:W_deriv_e}\\
    &W_{,z}(\varphi,\epsilonu,z) &&= \frac{1}{2}h'(z)\C(\epsilonu-\R\varphi) : (\epsilonu-\R\varphi)\label{eq:W_deriv_z}.
\end{alignat}

\subsection{Aim of the paper}\label{subsect:aim_of_the_paper}
The purpose of this work is to prove the existence of weak solutions to the problem \eqref{eq:problem}--\eqref{eq:initial_conditions}. To do so, we will introduce an appropriate time-discretised and regularised version of our system. Then, we will show that the discrete problem is well-posed and that its solution satisfies some a priori estimates. Finally, employing compactness results, we will pass to the limit as the time-step tends to $0$ and prove that the limit we find solves the original PDEs system.\\

\noindent \textbf{Mathematical difficulties.} The main mathematical challenges that we faced are the following. 
\begin{itemize}
       \item The presence of the mass source in the Cahn--Hilliard equation \eqref{eq:ch1}--\eqref{eq:ch2}, which implies that there is no mass conservation, i.e., the mean value of $\varphi$ is not constant. This is expected from the modelling point of view, however, it requires being able to handle the term 
       \begin{equation*}
           \into U(\varphi,\sigma,\epsilonu,z) \mu \dx
       \end{equation*}
       in the energy estimate (see the proof of \Cref{prop:a_priori_estimate}).
        \item The nonlinear coupling between the single equations. In particular, in the damage \cref{eq:damage}, the term
        \begin{equation*}
            W_{,z}(\varphi,\varepsilon(\uu),z)  =\frac{1}{2} h'(z) \C(\varepsilon(\uu)-\R \varphi):(\varepsilon(\uu)-\R \varphi)
        \end{equation*}
        is quadratic in $\varepsilon(\uu)$. In order to pass to the limit in this term from the discrete to the continuous problem, we have to perform a suitable regularity estimate for the displacement $\uu$ to obtain strong convergence for $\varepsilon(\uu)$. This estimate, in turn, requires a $L^{\infty}(0,T;W^{1,p})$ uniform bound for the damage $z$, with $p>d$: although the $p-$Laplacian operator in \cref{eq:damage} is a nonlinear operator which complicates the analysis, it has a fundamental regularising role. For the same reason, since we do not have uniform estimates for $\varphi$ in equally strong spaces, we cannot allow a dependence of the elasticity tensor on the phase. In the literature (see, e.g., \cite{heinemann_kraus_rocca_rossi_2017}) this issue has been addressed by putting a $\Delta_p \varphi$ instead of $\Delta\varphi$ in \cref{eq:ch2}. However, we will not follow this strategy here.
        \item The damage equation is highly nonlinear due to the presence of $-\Delta_p z$ and the subdifferential $\beta=\partial \widehat{\beta}(z)$. In particular, the $ p-$Laplacian operator seems to affect the possibility of gaining uniqueness due to its degenerate character. As already pointed out in \cite{Rocca_Rossi_2014} for a similar equation, this difficulty may be overcome by replacing the degenerate $p-$Laplacian operator $-\diver(|\nabla z|^{p-2} \nabla z)$ with the non-degenerate one $-\diver((1+|\nabla z|^2)^{\frac{p-2}{2}}\nabla z)$ or with the fractional $s$-Laplacian (see, e.g., \cite[][p. 1282]{Rocca_Rossi_2014} for a definition). However, we do not include such analysis in this paper and uniqueness remains an open problem. 
    \end{itemize}

\noindent \textbf{Plan of the paper.} The paper is organized as follows. In \Cref{sect:main_result}, after introducing some notation and preliminary results, we list the hypotheses under which we work. Then we state the weak formulation of our problem and our main result, i.e., \Cref{thm:existence}. \Cref{sect:proof} is completely devoted to the proof of the existence result.

\section{Main result}\label{sect:main_result}
\subsection{Notation and preliminaries}
\textbf{Notation.} In what follows, for any real Banach space $X$ with dual space $X'$, we indicate its norm as $\norm{\cdot}_{X}$ and the dual pairing between $X$ and $X'$ as $\duality{\cdot}{\cdot}_{X}$. We denote the Lebesgue and Sobolev spaces over $\Omega$ as $L^p \coloneqq L^p(\Omega)$, $W^{k,p} \coloneqq W^{k,p}(\Omega)$ and $H^k \coloneqq W^{k,2}(\Omega)$. We use $H^1_0$ to denote the functions of  $H^1$ that have zero trace at the boundary. We employ the notation $L^p_{\Gamma} \coloneqq L^p(\Gamma)$ for the Lebesgue spaces over $\Gamma$ and $\mathcal{H}^{d-1}$ for the $(d-1)$-dimensional Hausdorff measure. Moreover, to keep the notation as simple as possible, we will often not distinguish between scalar, vector, and matrix-valued spaces (for example, we will use $L^p$ instead of $L^p(\Omega)$ but also $L^p(\Omega;\RR^d)$). However, we will use bold font for vectors and calligraphic font for tensors. For the sake of brevity, the norm of the Bochner space $W^{k,p}(0,T;X)$ is indicated as $\norm{\cdot}_{W^{k,p}(X)}$, omitting the time interval $(0,T)$. Sometimes, for $p \in [1,+\infty)$, we will identify $L^p(Q)$ with $L^p(0,T;L^p)$. With the notation $C^0([0,T];X)$ we mean the space of continuous $X$-valued 
functions, while with $C_{\text{w}}^0([0,T];X)$ we mean the space of weakly continuous $X$-valued functions. Regarding the constants, as is common, we will use the notation $C$ to indicate a constant that depends only on the assigned data of the problem and whose value might change from line to line. If we want to highlight the dependency on a certain parameter, we put it as a subscript (e.g., $C_{\tau}$ indicates a constant that depends on $\tau$, $C_0$ a constant that depends on the initial data, and so on). \\

\noindent \textbf{Useful inequalities.} We will make use of classical inequalities such as H\"older, Young, Poincaré, and Poincaré--Wirtinger. For convenience, we recall a special case of the Gagliardo--Nirenberg interpolation inequality (see, e.g., \cite{Nirenberg_59}). 

\begin{thm}[Gagliardo--Nirenberg inequality]\label{thm:Gagliardo-Nieremberg_ineq}
    Let $\Omega \subseteq \RR^d$ be a Lipschitz bounded domain. Given 
    \begin{equation*}
        r > q \geq 1, \qquad s > d\left(\frac{1}{2}-\frac{1}{r}\right), \qquad \frac{1}{r} = \frac{\alpha}{q}+(1-\alpha)\left(\frac{1}{2}-\frac{s}{d}\right),
    \end{equation*}
    there exists a constant $C$ such as for every $v \in H^s$, the following inequality holds true:
    \begin{equation*}
        \norm{v}_{L^r} \leq C \norm{v}_{L^q}^{\alpha}\norm{v}_{H^s}^{1-\alpha}.
    \end{equation*}
\end{thm}

\noindent Another inequality we will employ is the following Ehrling's Lemma, also known as Aubin--Lions inequality (see \cite[][Theorem 16.4, p. 102]{Lions_Magenes_12}).

\begin{thm}[Ehrling's lemma]\label{thm:ehrling_lemma}
    Let $(X,\norm{\cdot}_X)$, $(Y,\norm{\cdot}_Y)$ and $(Z,\norm{\cdot}_Z)$ be Banach spaces with $X$ compactly embedded in $Y$ and $Y$ continuously embedded in $Z$.
    Then, for every $\varepsilon > 0$ there exists a $C(\varepsilon) > 0$ such that
    \begin{equation*}
        \norm{x}_{Y} \leq \varepsilon \norm{x}_{X} +C(\varepsilon)\norm{x}_{Z}
    \end{equation*}
    for every $x \in X$.
\end{thm}

\noindent Finally, a key part of our main result's proof is based on an estimate we obtain through a discrete version of the well-known Gronwall inequality. For the sake of completeness, we include its statement.

\begin{lemma}[Discrete Gronwall inequality]\label{lemma:discrete_gronwall}
    Let $\{x_n\}_{n\in\NN}$  be a  real sequence satisfying
    \begin{equation*}
        x_n \leq \theta+\sum_{k=0}^{n-1}\gamma_k x_k \qquad \forall n \in \NN
    \end{equation*}
    for a constant $\theta$ and a  non-negative sequence $\{\gamma_n\}_{n \in \NN}$. Then, it holds
    \begin{equation*}
        x_n \leq \theta \exp\bigg({\sum_{k=0}^{n-1} \gamma_k}\bigg)
    \end{equation*}
    for every $n \in \NN$.
\end{lemma}
\noindent A proof can be found in \cite{Clark_87}.\\

\noindent \textbf{Preliminary on mathematical visco-elasticity.} 
Let $\C=(c_{hijk})$ be a fourth-order tensor such that:
    \begin{enumerate}[(i)]
        \item $\C$ is symmetric, i.e., 
        \begin{equation}\label{cond_sym}
            c_{hijk}(x) = c_{ihjk}(x) = c_{jkhi}(x)
        \end{equation}
        for a.e. $x \in \Omega$ and for each index $i,j,k,h = 1, \dots, d$.
        \item $\C$ satisfies the strong ellipticity condition, i.e., there exists a positive constant $C$ such that for all $D \in \RR^{d \times d}_{\text{sym}}$ and for a.e. $x \in \Omega$
        \begin{align} \label{cond_ellipt}
        \C(x) D : D \geq C |D|^2,
        \end{align} 
        where $:$ denotes the standard Frobenius inner product between matrices.
    \end{enumerate}
As will be specified in \Cref{subsect:hyp}, both the elasticity tensor $\C$ and the viscosity tensor $\V$ in \cref{eq:displacement} satisfy \eqref{cond_sym} and \eqref{cond_ellipt}. 
Moreover, the following regularity result holds true.
\begin{lemma} \label{lemma:H2_inequality}
    Let $\Omega$ be a $C^2$ domain in $\Rd$ and  $\C=(c_{hijk}) \in W^{1,\infty}(\Omega;\RR^{d \times d \times d \times d})$ be a symmetric and strongly elliptic fourth-order tensor. 
    Then, there exist $C_1, C_2 > 0$ such that for every $\uu$ in $H^2$  with $\uu = \mathbf{0}$ on $\partial \Omega$ 
    \begin{equation*}
        C_1 \norm{\uu}_{H^2} \leq \norm{\diver[\C \epsilonu]}_{L^2} \leq C_2 \norm{\uu}_{H^2}.
    \end{equation*}
\end{lemma}
\noindent For more details, cf. \cite[][Proposition 1.5, p. 318]{marsden1994} and \cite[][Lemma 3.2., p. 263]{necas2012}.\\

\noindent \textbf{Subdifferentials of convex functions.} Here, we will introduce some notation and recall some facts about a class of maximal monotone operators. Given a proper, convex, and lower semicontinuous function $\phi : \RR \to (-\infty,+\infty]$, its realisation in $L^2(\Omega)$ is a proper, convex, and lower semicontinuous function $\Phi: L^2(\Omega) \to (-\infty,+\infty]$ defined naturally as
\begin{equation*}
     \Phi(v) \coloneqq \begin{cases}
    \into \phi(v) \dx  & \text{ if } \phi(v) \in L^{1}(\Omega),\\
        +\infty &\text{ otherwise,}
    \end{cases}
\end{equation*}
for every $v \in L^2(\Omega)$. It is well known that the subdifferentials of $\phi$ and $\Phi$ are maximal monotone operators. Moreover, it holds that $\xi \in \partial \Phi(v)$ if and only if $\xi(x) \in \partial \phi(v(x))$ for a.e. $x \in \Omega$. Furthermore, the Moreau--Yosida approximations of parameter $\delta$ of the previous functions are linked by the following relation:
\begin{equation*}
    \Phi_{\delta}(v) = \into \phi_{\delta}(v) \dx.
\end{equation*}
In light of all these properties, with a slight abuse of notation, we will write $\phi$ instead of $\Phi$ and $\partial \phi$ instead of $\partial \Phi$. For more details, the interested reader may refer to \cite[][Proposition 2.16, p. 47]{brezis1973}.\\

\noindent \textbf{The $p$-Laplace operator with homogeneous Neumann conditions.} Let $p \geq 2$ and define
\begin{equation*}
    \Phi_p: L^2(\Omega) \to [0,+\infty], \qquad \Phi_p(v) \coloneqq \begin{cases}
        \frac{1}{p}\into |\nabla v|^p \dx & \text{ if } v \in W^{1,p}(\Omega),\\
        +\infty &\text{ otherwise.}
    \end{cases}
\end{equation*}
Then $\Phi_p$ has domain $\mathcal{D}(\Phi_p) = W^{1,p}(\Omega)$ and it is proper, convex and lower semicontinuous on $L^2(\Omega)$. Hence, its subdifferential $-\Delta_p \coloneqq \partial \Phi_p$ is a maximal monotone operator. Moreover, every $v$ in the natural domain
\begin{equation*}
    \mathcal{D}(-\Delta_p)= \{ v \in W^{1,p}(\Omega) \,:\, -\Delta_p v \in L^2(\Omega), \quad  |\nabla v|^{p-2} \nabla v \cdot \nnu = 0 \text{ on } \Gamma\},
\end{equation*}
it satisfies
\begin{equation*}
     \into -\Delta_p v \, \omega \dx = \into |\nabla v |^{p-2}\nabla v  \cdot \nabla \omega \dx 
\end{equation*}
for every $\omega \in \mathcal{D}(\Phi_p)$. In particular, 
\begin{equation*}
    -\Delta_p v = -\diver(|\nabla v|^{p-2}\nabla v)
\end{equation*}
in the sense of distributions. 
Finally, the following regularity result holds true (the interested reader can refer to \cite[][Theorem 2, Remark 3.5]{Savare98}). 
\begin{lemma}\label{lemma:regularity_inequality} 
    For all $0 < \delta < \frac{1}{p}$, the inclusion $\mathcal{D}(-\Delta_p) \subseteq W^{1+\delta,p}(\Omega)$ holds. Moreover, it exists $C_{\delta} > 0$ such that, for all $v \in W^{1+\delta,p}(\Omega)$,
    \begin{equation*}
        \norm{v}_{W^{1+\delta,p}} \leq C_{\delta}( \norm{-\Delta_p v }_{L^2} + \norm{v}_{L^2}).
    \end{equation*}
\end{lemma}

\subsection{Hypotheses}\label{subsect:hyp}
Let $d=2,3$ denote the space dimension and $\Omega$ a bounded $C^2$-domain in $\RR^d$.

\begin{enumerate}[\rm{(H\arabic*)}]
    \item \label{hyp:sources} 
    Regarding the nonlinear sources $U$ and $S$ defined in \eqref{eq:defn_U} and \eqref{eq:defn_S}, we consider
    \begin{align}
        & \lambda_p,\ \lambda_a,\ \lambda_c\ \text{ non-negative constants},\\
        & g \in C^0(\RR^2), \text{ non-negative and bounded},\\
        & f \in L^{\infty}(0,T;L^2),\\ 
        & \Lambdac \in C^0(\RR), \text{ non-negative and bounded},\\ 
        & \sigma_c \in L^{\infty}(Q)\text{, non-negative}.
    \end{align}
    
    \item \label{hyp:psi}
    Regarding the smooth potential $\Psi \in C^1(\RR)$, we suppose that the following growth conditions hold
    \begin{align}
        &\Psi(r) \geq C_1|r|^2-C_2, \label{eq:property_psi}\\
        &|\Psi'(r)| \leq C_3 \Psi(r) + C_4 \label{eq:property_psi_derivative}
    \end{align}
    for some fixed positive constants $C_1$, $C_2$, $C_3$, $C_4$ and for every $r \in \RR$.\\
    Moreover, we assume that there exists a convex-concave splitting $\Psi = \psiu + \psin$ such that
    \begin{align}
        &\psiu, \psin \in C^1(\RR),\label{eq:hyp_psi_reg}\\
        &\psiu \text{ is convex and its derivative satisfies } \psiu'(0)=0, \label{eq:hyp_psi_convex}\\
        &\psin \text{ is concave}, \label{eq:hyp_psi_concave}\\
        &\psin' \text{ is  Lipschitz continuous}. \label{eq:hyp_psi_concave_lip}
    \end{align}
\end{enumerate}

\begin{remark}
     We point out that requiring the nonconvex part of the decomposition to be concave is not restrictive. In fact, for every $\psiu,  \widetilde{\Psi} \in C^1(\RR)$ such that $\Psi = \psiu + \widetilde{\Psi}$, where $\psiu$ satisfies \eqref{eq:hyp_psi_convex} and $\widetilde{\Psi}'$ satisfies \eqref{eq:hyp_psi_concave_lip} with a Lipschitz constant $L$, we can consider 
    \begin{equation*}
        \Psi(r) = \left( \psiu(r) + \frac{L}{2} r^2\right) + \left( \widetilde{\Psi}(r)- \frac{L}{2}r^2\right)
    \end{equation*}
    for every $r \in \RR$, which is compliant to \eqref{eq:hyp_psi_reg}--\eqref{eq:hyp_psi_concave_lip}.
\end{remark}

\begin{remark} \label{remark:phi_physical_interval}
    Note that the Hypothesis \ref{hyp:psi} is compatible with the classical choice $\Psi(r)=\frac{1}{4}(1-r^2)^2$. However, it does not allow us to consider singular potential, such as logarithm type. This means that we can not guarantee that $\varphi$ has values in the physically relevant interval $[-1,1]$.
\end{remark}

\begin{enumerate}[\rm{(H\arabic*)},resume]
    \item \label{hyp:elastic_viscous_tensors}
    We assume that the fourth-order elasticity tensor $\C$ in \eqref{eq:elastic_energy_expression} belongs to the space $C^1(\Omega; \RR^{d \times d \times d \times d})$ and is
    \begin{align}
        &\text{Lipschitz continuous and bounded,}\\
        &\text{symmetric (i.e., it satisfies \eqref{cond_sym})},\\
        &\text{strongly elliptic (i.e., it satisfies \eqref{cond_ellipt})}.
    \end{align}
    Regarding the fourth-order viscous tensor $\V$, we suppose that it is of the form 
    \begin{equation}\label{eq:V_form}
        \V = \omega \C
    \end{equation}
    for a positive constant $\omega$.
\end{enumerate}

\begin{remark}
    It is worth pointing out that the viscosity tensor is usually assumed to be only symmetric and positively defined. The stronger assumption \eqref{eq:V_form} is made in order to prove the desired regularity for the displacement $\uu$. Without it, our argument does not apply anymore  (see the proof of \Cref{prop:discrete_existence} below).
\end{remark}

\begin{enumerate}[\rm{(H\arabic*)},resume]
    \item \label{hyp:elastic_viscous_coeff}
    We require that the scalar function $h$ in \eqref{eq:elastic_energy_expression} is $C^2(\RR)$ and that 
    \begin{align}
        &h \text{ and  } h' \text{ are Lipschitz continuous},\\
        &h \text{ is bounded with }0 \leq h \leq h^*.
    \end{align}
    We postulate that the viscosity coefficient $a$ is $C^{1}(\RR)$ and that it satisfies
    \begin{align}
        & a \text{ is Lipschitz continuous,}\\
        & a \text{ is bounded with } 0 <a_* \leq  a \leq a^*.
    \end{align}
    
    \item \label{hyp:p}
    We assume that the constant $p$ that occurs in the $p$-Laplacian $-\Delta_p$ in the damage equation \eqref{eq:damage} satisfies
    \begin{equation}
        p > d
    \end{equation}
    where $d$ is the space dimension.
  
    \item \label{hyp:pi}
    We consider a function $\widehat{\pi} \in C^1(\RR)$ with derivative $\pi \coloneqq \widehat{\pi}'$ that satisfies
    \begin{align}
        &\pi \text{ is Lipschitz continuous}.
    \end{align}
    
    \item \label{hyp:beta}
    Let $\widehat{\beta}: \RR \to [0,+\infty]$ be a function
    \begin{align}
        &\text{proper, convex and lower semicontinuous}\\
        & \text{with }  \interior(\mathcal{D}(\widehat{\beta})) \neq \emptyset,
    \end{align}
    and denote by $\beta \coloneqq \partial \widehat{\beta} : \RR \rightrightarrows \RR$ its subdifferential.

\end{enumerate}

\begin{remark}\label{remark:z_physical_interval}
    Note that Hypothesis \ref{hyp:beta} is quite general, and is compatible with a large class of potentials.  A simple and classical example to keep in mind is the following
    \begin{equation*}
        \widehat{\beta}(r)=I_{[0,1]}(r)= \begin{cases}
        0 & \text{if } r \in [0,1],\\
        +\infty & \text{otherwise}.
    \end{cases}
    \end{equation*}
    In particular, it would ensure that the damage $z$ has values in the physically significant range $[0,1]$.
\end{remark}

\begin{enumerate}[\rm{(H\arabic*)},resume]
    \item \label{hyp:boundary_conditions}
    Regarding the boundary conditions \eqref{eq:sigma_nu} for the nutrient, we assume that
    \begin{align}
        &\sigma_{\Gamma} \in L^{\infty}(\Sigma) \text{ and } \sigma_{\Gamma} \geq 0,\\
        & \alpha \geq 0. 
    \end{align}
    
    \item \label{hyp:initial_conditions}
    Regarding the initial conditions \eqref{eq:initial_conditions}, we assume that
    \begin{align}
        &\varphi_0 \in H^1,\ \Psi(\varphi_0) \in L^1,\\
        &\sigma_0 \in L^{2},\ 0 \leq \sigma_0 \leq M \coloneqq \max\{ \norm{\sigma_c}_{L^{\infty}(Q)}, \norm{\sigma_{\Gamma}}_{L^{\infty}(\Sigma)} \}, \label{eq:hp_sigma0}\\
        &\uu_0 \in H^2 \cap H^1_0,\ \vv_0 \in H^1_0,\\
        &z_0 \in \mathcal{D}(-\Delta_p),\ \widehat{\beta}(z_0) \in L^1. 
    \end{align}
\end{enumerate}

\subsection{Weak formulation and existence result}
\begin{defn}\label{defn:weak_solution}
    We say that a quintuplet $(\varphi,\, \mu,\, \sigma,\, \uu,\, z)$ is a weak solution to the PDEs system \eqref{eq:problem}--\eqref{eq:initial_conditions} if it has the regularity
    \begin{gather*}
            \varphi \in  L^{2}(0,T;H^2) \cap L^{\infty}(0,T; H^1) \cap H^1(0,T;(H^1)'),\quad  \mu \in L^2(0,T;H^1),\\
            \sigma \in L^2(0,T;H^1) \cap H^1(0,T;(H^1)'),\\
            \uu \in H^1(0,T;H^2\cap H^1_0) \cap W^{1,\infty}(0,T;H^1_0) \cap H^2(0,T;L^2),\\
            z \in L^{\infty}(0,T;W^{1,p}) \cap H^1(0,T;L^2),\, -\Delta_p z \in L^2(0,T;L^2) 
    \end{gather*}
    and there exists a subgradient
    \begin{equation*}
        \xi \in L^2(0,T;L^2) \text{ with } \xi \in \beta(z) \text{ a.e. in } Q,
    \end{equation*}
    such that the following equations are satisfied a.e. in $(0,T)$
    \begin{subequations}\label{eq:problem_eq_with_spaces}
        \begin{align}
            &\duality{\partial_t \varphi}{\zeta}_{H^1} + \into  \nabla \mu \cdot \nabla \zeta \dx = \into U(\varphi,\sigma,\epsilonu,z) \, \zeta \dx ,\\
            & \into \mu \zeta \dx  = \into \nabla \varphi\cdot \nabla \zeta \dx + \into \Psi'(\varphi) \zeta \dx + \into W_{,\varphi}(\varphi, \epsilonu, z) \zeta \dx, \label{eq:ch2_def}\\
            &\duality{\partial_t\sigma}{\zeta}_{H^1} + \!\!\into \!\nabla \sigma \cdot \nabla \zeta \dx + \alpha \!\intg (\sigma - \sigma_{\Gamma}) \zeta \dA = \!\into \!S(\varphi,\sigma,z) \zeta \dx ,\\
            & \into \partial_{tt}\uu \cdot \boldsymbol{\omega} \dx + \into \left[a(z)\V \varepsilon(\partial_t\uu) + W_{,\varepsilon}(\varphi, \epsilonu, z)  \right] : \varepsilon(\boldsymbol{\omega}) \dx  = 0,\label{eq:dispalcement_def}\\
            & \begin{aligned}
                \into \partial_t z \rho \dx + \into |\nabla z|^{p-2} \nabla z \cdot \nabla \rho \dx &+ \into \xi \rho \dx+ \into\pi(z)\rho \dx  \\
                &+ \into W_{,z}(\varphi, \epsilonu, z) \rho \dx = 0 
            \end{aligned}\label{eq:damage_def}
        \end{align}
    \end{subequations}
for all $\zeta \in H^1$, $\boldsymbol{\omega} \in H^1_0$ and $\rho \in W^{1,p}$. Moreover, we require that the quintuplet complies with the initial conditions, i.e.,
\begin{equation*}
    \varphi(0)=\varphi_0, \quad \sigma(0)=\sigma_0, \quad  \uu(0)=\uu_0, \quad \partial_t\uu(0)=\vv_0, \quad z(0)=z_0 \quad \text{ a.e. in } \Omega.
\end{equation*}
\end{defn}
\begin{remark}\label{remark:eq_formulation}
    Notice that, with the regularity we demand, requiring \eqref{eq:ch2_def} is equivalent to asking that \cref{eq:ch2} is satisfied in $L^2(0,T;L^2)$ and the boundary condition $\nabla \varphi \cdot \nnu=0$ in \eqref{eq:boundary_conditions} is satisfied in the sense of the traces. The same holds also for the damage equation. Similarly, \cref{eq:dispalcement_def} is equivalent to ask that
    \begin{multline*}
        \partial_{tt}\uu -a'(z)\V\varepsilon(\partial_t\uu)\nabla z - a(z)\diver\left[\V\varepsilon(\partial_t\uu)\right] \\- h'(z)\C\left( \epsilonut -\R\varphi\right) \nabla z - h(z)\diver\left[\C\epsilonut - \C\R\varphi \right]= \mathbf{0}
    \end{multline*}
    is satisfied in $L^2(0,T;L^2)$ and that the boundary condition $\uu = \mathbf{0}$ in \eqref{eq:boundary_conditions} holds in the sense of the traces.
\end{remark}

\begin{remark} Note that, by standard embedding results (see \cite{Strauss_66} and \cite{Lions_Magenes_12}), 
\begin{align*}
    & \varphi \in L^{\infty}(0,T;H^1) \cap C^0([0,T]; (H^1)') &&\mkern-45mu\hookrightarrow \mkern-45mu &&C^0_{\text{w}}([0,T];H^1),\\
    &\sigma \in L^2(0,T;H^1) \cap H^1(0,T; (H^1)') &&\mkern-45mu\hookrightarrow \mkern-45mu && C^0([0,T];L^2),\\
    &\uu \in H^1(0,T;H^2 \cap H^1_0) &&\mkern-45mu\hookrightarrow \mkern-45mu && C^0([0,T];H^2 \cap H^1_0),\\
    &\partial_t\uu \in L^{\infty}(0,T;H^1_0) \cap C^0([0,T]; L^2) &&\mkern-45mu\hookrightarrow \mkern-45mu && C^0_{\text{w}}([0,T];H^1_0),\\
    & z \in L^{\infty}(0,T;W^{1,p}) \cap C^0(0,T; L^2) &&\mkern-45mu\hookrightarrow \mkern-45mu && C^0_{\text{w}}([0,T];W^{1,p}),
\end{align*}
so $\varphi(0)$ makes sense in $H^1$,  $\sigma(0)$ in $L^2$,  $\uu(0)$ in $H^2 \cap H^1_0$, $\partial_t \uu(0)$ in $H^1_0$ and $z(0)$ in $W^{1,p}$. This justifies the initial data regularities that we prescribed.     
\end{remark}

\begin{thm}\label{thm:existence}
Let Hypotheses \ref{hyp:sources}--\ref{hyp:initial_conditions} be satisfied. Then, there exists a weak solution to the system \eqref{eq:problem}--\eqref{eq:initial_conditions} in the sense of \Cref{defn:weak_solution}
with the additional property that 
\begin{equation*}
    0 \leq \sigma \leq M \text{  a.e. in } Q.
\end{equation*}
\end{thm}

\begin{remark}
    It is not difficult to prove that if $(\varphi, \mu, \sigma, \uu, z)$ is a weak solution to the PDEs system \eqref{eq:problem}--\eqref{eq:initial_conditions}, then $\varphi$ enjoys the maximal regularity
    \begin{equation}\label{eq:phi_max_reg}
        \varphi \in L^4(0,T; H^2).
    \end{equation}
    Here, we show \eqref{eq:phi_max_reg} by following the approach from \cite[][Remark 2.3, p. 1562]{Garcke_Lam_Signori_Opt_Contr_2021}. As pointed out in \Cref{remark:eq_formulation}, due to the regularity and boundary conditions fulfilled by $\varphi$, equation \eqref{eq:ch2} is satisfied in $ L^2$ a.e. in $(0,T)$. Explicitly, it holds
    \begin{equation*}
        \into \mu \zeta \dx = \into \left[(-\Delta \varphi)  + \Psi'(\varphi) + W_{,\varphi}(\varphi, \epsilonu, z) \right] \zeta \dx
    \end{equation*}
    for every $\zeta \in L^2$, a.e. in $(0,T)$. Taking $\zeta = - \Delta \varphi$ as a test function, we have
    \begin{equation*}
        \begin{split}
            \norm{-\Delta \varphi}_{L^2}^2 &= \into \nabla \mu \cdot \nabla \varphi \dx - \into \Psi'(\varphi) (-\Delta \varphi) \dx - \into W_{,\varphi}(\varphi, \epsilonu, z)(- \Delta \varphi) \dx\\
            & \coloneqq I_1 + I_2 + I_3,
        \end{split}
    \end{equation*}
    where we have integrated by parts in the first term on the right-hand side and employed homogeneous Neumann boundary conditions. To handle $I_1$, we simply use the H\"older inequality. Regarding $I_2$, we write $\Psi$ as the sum of its convex and concave parts. Then, proceeding formally, we observe that
    \begin{equation*}
        \into \psiu'(\varphi)(-\Delta \varphi) \dx = \into \psiu''(\varphi)|\nabla \varphi|^2 \dx \geq 0
    \end{equation*}
    since $\psiu$ is convex. Notice that this is not rigorous because $\psiu$ is only $C^1$, but this inequality can be proved employing the Yosida--Moreau approximation of $\psiu$, as we will do in details in the proof of \Cref{prop:a_priori_estimate}. Thus, we have
    \begin{equation*}
        I_2 = - \into [\psiu'(\varphi)+ \psin'(\varphi) ] (-\Delta \varphi) \dx \leq  - \into \psin''(\varphi)|\nabla \varphi|^2 \dx  \leq C \norm{\nabla \varphi}_{L^2}^2
    \end{equation*}
    because $\psin'$ is Lipschitz continuous accordingly to hypothesis \ref{hyp:psi}. Finally, we turn our attention to $I_3$. Applying the H\"older and the Young inequalities leads to 
    \begin{equation*}
        I_3 \leq \norm{W_{,\varphi}(\varphi, \epsilonu, z)}_{L^2} \norm{-\Delta \varphi}_{L^2} \leq \frac{1}{2} \norm{W_{,\varphi}(\varphi, \epsilonu, z)}_{L^2}^2 + \frac{1}{2} \norm{-\Delta \varphi}_{L^2}^2.
    \end{equation*}
    Then, the term related to $W_{,\varphi}$ can be treated as follows:
    \begin{equation*}
        \norm{W_{,\varphi}(\varphi, \epsilonu, z)}_{L^2}^2 = \into \left|h(z)\C(\epsilonu-\R \varphi) : \R  \right|^2 \dx \leq C \left(\norm{\epsilonu}_{L^2}^2 + \norm{\varphi}_{L^2}^2\right),
    \end{equation*}
    recalling that $h$ is bounded by hypothesis \ref{hyp:elastic_viscous_coeff}. Putting these estimates together, we obtain
    \begin{equation*}
        \frac{1}{2}\norm{-\Delta \varphi}_{L^2}^2 \leq C \left( \norm{\nabla \mu}_{L^2} \norm{\nabla \varphi}_{L^2} + \norm{\varphi}_{H^1}^2 + \norm{\epsilonu}_{L^2}^2\right) \leq  C \left( \norm{\nabla \mu}_{L^2}  + 1\right),
    \end{equation*}
    where the last inequality holds because
    \begin{equation*}
        \varphi \in L^{\infty}(0,T; H^1), \qquad \uu \in W^{1,\infty}(0,T;H^1_0).
    \end{equation*}
    Taking the square of both sides and integrating in time, we end up with
    \begin{equation*}
        \norm{-\Delta \varphi}^4_{L^4(L^2)} \leq C \left( \norm{\nabla \mu}_{L^2(L^2)}^2 + 1\right) \leq C,
    \end{equation*}
    since $\mu \in L^2(0,T;H^1)$. Thus, \eqref{eq:phi_max_reg} follows from standard elliptic regularity.
\end{remark}

\section{Proof of the existence theorem}\label{sect:proof}
To prove the existence theorem, we will introduce a semi-implicit Euler scheme that is a time-discrete and regularised version of our system. 

\subsection{Time discretisation} \label{subsection:time_discretisation}
Let $\tau$ be a positive and small real number. We consider a partition of $[0,T]$ with nodes 
\begin{equation*}
    t_{\tau}^k \coloneqq \begin{cases}
        k \tau & k=0, \dots, K_{\tau}-1,\\
        T & k = K_{\tau},
    \end{cases}
\end{equation*}
 where $K_{\tau}$ is the ceiling integer part of $T / \tau$, i.e., is the greater integer such that $\tau (K_{\tau}-1)$ is strictly smaller than $T$. We also introduce the notation: 
\begin{equation*}
    I_{\tau}^k \coloneqq \begin{cases}
        [0,\tau] &\text{ if } k=1,\\
        (t_{\tau}^{k-1},\ t_{\tau}^k] &\text{ if } k = 2, \dots, K_{\tau}.
    \end{cases}
\end{equation*}
With a slight abuse of terminology, we refer to the partition as uniform and to $\tau$ as its time step, even though the last interval may have a smaller length than $\tau$.
We approximate $f$, $\sigma_c$ and $\sigma_{\Gamma}$  with their local means, i.e., we define
\begin{equation*}
    \fk \coloneqq \frac{1}{\tau}\intk f\ds, \qquad  \sigmack \coloneqq \frac{1}{\tau}\intk \sigma_c\ds,\qquad  \sigmagk \coloneqq \frac{1}{\tau}\intk \sigma_{\Gamma} \ds,  
\end{equation*}
for every $k=1, \dots, K_{\tau}$.
\begin{remark}
It is obvious that, since $f \in L^{\infty}(0,T;L^2)$, $\sigma_c \in L^{\infty}(Q)$, and $\sigma_{\Gamma} \in L^{\infty}(\Sigma)$, then $\fk \in  L^{2}$, $\sigmack \in  L^{\infty}$, and $\sigmagk \in L^{\infty}_{\Gamma}$ with 
\begin{equation} \label{eq:estimate_dataK}
    \norm{\fk}_{L^{2}} \leq \norm{f}_{L^{\infty}(L^{2})},\ \norm{\sigmack}_{L^{\infty}} \leq \norm{\sigma_c}_{L^{\infty}(Q)},\
    \norm{\sigmagk}_{L^{\infty}_{\Gamma}} \leq \norm{\sigma_{\Gamma}}_{L^{\infty}(\Sigma)},
\end{equation}
for every $k=1, \dots, K_{\tau}$. In addition, $0 \leq \sigmack, \sigmagk \leq M$. 
\end{remark}

\noindent For every sequence of scalar or vector-valued functions $\{w_k\}_k$ defined over $\Omega$, we adopt the notation:
\begin{equation*}
    \Dk w = \frac{w_k-w_{k-1}}{\tau},\qquad \Dk^2 w = \frac{w_k-2w_{k-1} + w_{k-2}}{\tau^2},
\end{equation*}
for every $k$ for which it makes sense. We introduce the time-discrete approximation of our problem, which is posed in $\Omega$:
\begin{subequations}\label{eq:problemK}
    \begin{align}
        &\Dk\varphi - \Delta \muk = U_k -\tau\Dk\mu 
       \label{eq:ch1_k},\\
        & \muk = -\Delta \phik + \psiu'(\phik) + \psin'(\phikm) + W_{,\varphi}(\phik, \epsilonukm, \zkm) + \tau \Dk \varphi\label{eq:ch2_k},\\
        & \Dk\sigma - \Delta \sigmak = S_k\label{eq:nutrient_k},\\
        & \Dk^2\uu - \diver\left[a(\zk)\V \epsilonudk  + W_{,\varepsilon}(\phik, \epsilonuk, \zk)  \right] = \mathbf{0} \label{eq:u_k},\\
        &\begin{aligned}
            \Dk z - \Delta_p \zk + \beta_{\tau}(\zk) + \pi(\zkm)  &+ \Wzu(\phik, \epsilonukm, \zk)\\ &+ \Wzn(\phik, \epsilonukm, \zkm) = 0.
        \end{aligned}\label{eq:z_k}
    \end{align}
\end{subequations}
Here, for brevity, we employed the following notation for the source terms:
\begin{align*}
    & U_k \coloneqq \bigg(\frac{ \lambda_p \sigmak}{1+|W_{,\varepsilon}(\phikm,\epsilonukm,\zkm)|} - \lambda_a + \fk\bigg) g(\phikm,\zkm),\\
    & S_k \coloneqq  -\lambda_c\sigmak g(\phikm,\zkm) +\Lambdac(\zkm)(\sigmack-\sigmak).
\end{align*}
The system \eqref{eq:problemK} is coupled with the boundary conditions on $\Gamma$:
\begin{subequations}\label{eq:boundary_conditionsk}
    \begin{align}
        & \nabla\phik \cdot \nnu = \nabla \muk \cdot \nnu = 0, \label{eq:phi,mu_nuk}\\
        & \nabla\sigmak \cdot \nnu + \alpha(\sigmak -\sigmagk) = 0, \label{eq:sigma_nuk}\\
        & \uk = \mathbf{0},  \label{eq:u_nuk}\\
        & (|\nabla \zk|^{p-2}\nabla \zk) \cdot \nnu = 0.\label{eq:z_nuk}
    \end{align}
\end{subequations}
For every $\tau > 0$  we employ a recursive procedure that, starting from the initial values 
\begin{equation} \label{eq:defn_mu0}
    \varphi_{\tau}^0 \coloneqq \varphi_0, \quad
    \sigma_{\tau}^0 \coloneqq \sigma_0,\quad 
    \uu_{\tau}^0 \coloneqq \uu_0,\quad
    z_{\tau}^0 \coloneqq z_0,\quad
\end{equation}
gives $(\phik,\ \muk,\ \sigmak,\ \uk,\ \zk)$ for every $k = 1, \dots, K_{\tau}$ that satisfies the previous system \eqref{eq:problemK}--\eqref{eq:boundary_conditionsk} in a proper sense that will be specified in \Cref{prop:discrete_existence}. 
Notice that, due to the presence of $\tau (\Dk \mu) = \muk - \mukm$ in the discrete Cahn--Hilliard equation \eqref{eq:ch1_k}, at the step $k=1$ the term $\mu_{\tau}^0$ appears. So, we define 
\begin{equation*}
    \mu_{\tau}^0 \coloneqq 0.
\end{equation*}
Similarly, to give a meaning to the term $D^2_{\tau,k}\uu$ in the displacement equation \eqref{eq:u_k} at the step $k=1$, we introduce
\begin{equation*}
    \uu^{-1}_{\tau} \coloneqq \uu_0 - \tau \vv_0,
\end{equation*}
where $\uu_0$ and $\vv_0$ are, respectively, the initial displacement and the initial velocity prescribed in \eqref{eq:initial_conditions}.
 Moreover,  we will sometimes denote the time-discrete velocity at the time-step $k$ as
\begin{equation*}
    \vk \coloneqq \Dk \uu.
\end{equation*}
Before stating the well-posedness result for the approximate system, let us comment briefly on our discretisation scheme. 
\begin{itemize}
    \item Regarding the discrete Cahn--Hilliard equation \eqref{eq:ch1_k}--\eqref{eq:ch2_k}, we added the regularising terms $-\tau\Dk \mu = \mukm - \muk$ and $\tau\Dk \varphi= \phik - \phikm$ respectively to \eqref{eq:ch1_k} and \eqref{eq:ch2_k}. Without going into technical details now, let us just say that $-\tau\Dk \mu$ allows us to rewrite the equations \eqref{eq:ch1_k}--\eqref{eq:ch2_k} in an equivalent abstract form for which the existence of a solution is automatically guaranteed. This formulation is obtained thanks to the term $(I-\Delta)\muk$ that appears in the equation \eqref{eq:ch1_k}. Thus, we can apply the inverse of $(I-\Delta)$ that, as we will see, has some good properties, obtain $\muk$ and substitute it in equation \eqref{eq:ch2_k}. On the other hand, thanks to $-\tau\Dk \varphi$, the term $\phik$ appears in the equation \eqref{eq:ch2_k}. It guarantees some coercivity and ensures the uniqueness of the solution. Notice that both terms $\Dk \mu$ and $\Dk \varphi$ are multiplied by $\tau$, so they are expected to vanish as $\tau \to 0$. The second choice we made is to evaluate $\psiu$ at $\phik$ and $\psin$ at $\phikm$. This is quite common and, again, motivated by some solvability issues. The main idea is to exploit the monotonicity of $\psiu'$ to prove existence and the fact that $\psin'$ is Lipschitz continuous to control the $L^2$- norm of this perturbative term.
    
    \item We employed a convex-concave splitting for $W$ with respect to its third variable
    \begin{align*}
        &\oset{$\smallsmile$}{W}{}_3(\varphi, \epsilonu, z) \coloneqq \frac{1}{2}\hu(z)\C(\epsilonu-\R\varphi) : (\epsilonu-\R\varphi),\\
        &\oset{$\smallfrown$}{W}{}_3(\varphi, \epsilonu, z) \coloneqq \frac{1}{2}\hn(z)\C(\epsilonu-\R\varphi) : (\epsilonu-\R\varphi),
    \end{align*}
    which, in turn, relies on a convex-concave splitting for $h$ given by
    \begin{equation*}
        \hu(z) \coloneqq h(z) + \frac{1}{2} \left(\sup_{x \in \RR} |h''(x)|\right) z^2,\qquad \hn(z) \coloneqq -\frac{1}{2} \left(\sup_{x \in \RR} |h''(x)|\right) z^2.
    \end{equation*}
    We observe that $\hu$ is convex, $\hn$ is concave, and  $h=\hu+\hn$. Consequentially, $\oset{$\smallsmile$}{W}{}_3$ is convex, $\oset{$\smallfrown$}{W}{}_3$ is concave, and $W = \oset{$\smallsmile$}{W}{}_3 + \oset{$\smallfrown$}{W}{}_3$.  It is worth pointing out that, since $h'$ is Lipschitz by Hypothesis \ref{hyp:elastic_viscous_coeff}, the same holds for $\hu'$ and $\hn'$ and, since $h > 0$, also $\hu > 0$. However, $\hu$ and $\hn$ are not bounded. Notice that we did not need to introduce a convex-concave decomposition for $W$ with respect to its first and second variables because it is already convex with respect to  $\varphi$ and $\epsilonu$. This splitting, as well as the careful choice between implicit and explicit arguments for the derivatives of $W$, will have a key role in carrying out the discrete energy a priori estimate in \Cref{prop:a_priori_estimate}, where we will employ the following trivial result, the proof of which is just a simple application of convex and concave inequalities.
    \end{itemize}
    
\begin{lemma}\label{convex_concave_decomposition}
Let $F: \RR \to \RR$ be a differentiable function that admits a convex-concave decomposition $F = \oset{$\smallsmile$}{F}{} + \oset{$\smallfrown$}{F}{}$ with differentiable $\oset{$\smallsmile$}{F}{}$ and $\oset{$\smallfrown$}{F}{}$. Then, 
\begin{equation*}
    (\oset{$\smallsmile$}{F}{}'(x) + \oset{$\smallfrown$}{F}{}'(y))(x-y) \geq F(x) - F(y)
\end{equation*}
for every $x,y \in \RR$.
\end{lemma}

\begin{itemize}
    \item Finally, we replaced $\widehat{\beta}$ with its Moreau--Yosida approximation $\widehat{\beta}_{\tau}$ defined by
    \begin{equation*}
        \widehat{\beta}_{\tau}(z) \coloneqq \min_{y \in \RR} \bigg\{\frac{1}{2\tau}|y-z|^2+\widehat{\beta}(y )\bigg\} \qquad \forall z \in \RR,
    \end{equation*}
    and, consequentially, the maximal monotone operator $\beta$ in the damage equation with $\beta_{\tau} \coloneqq (\widehat{\beta}_{\tau})'$. Note that we set the regularisation parameter equal to the time step $\tau$ so that we will pass to the limit simultaneously in the Yosida regularisation and in the time discretisation as $\tau \to 0$.
\end{itemize}
\begin{remark}
    We recall that $\widehat{\beta}_{\tau} \in C^1(\RR)$ is still convex and that $\beta_{\tau}$ is non-decreasing and Lipschitz continuous with Lipschitz constant bounded by $\tau^{-1}$ (see \cite[][Proposition 2.6, p. 28 and Proposition 2.11 p. 39]{brezis1973}). Moreover, since $\widehat{\beta}$ is non-negative, $\widehat{\beta}_{\tau}$ is non-negative. Finally, it is obvious by the definition of Moreau--Yosida approximation that $\widehat{\beta}_{\tau}(z) \leq \widehat{\beta}(z)$ for every $z \in \RR$.
\end{remark}

\begin{prop} \label{prop:discrete_existence}
    Let Hypotheses \ref{hyp:sources}--\ref{hyp:initial_conditions} be satisfied. Then, for every $k = 1 \dots K_{\tau}$, there exists a unique weak solution 
    \begin{equation*}
        (\phik,\, \muk,\, \sigmak,\, \uk, \zk) \in H^2 \times H^2 \times H^1 \times H^2 \times \mathcal{D}(-\Delta_p)
    \end{equation*}
    to the system \eqref{eq:problemK}--\eqref{eq:boundary_conditionsk} in the sense that
    it satisfies the boundary conditions \eqref{eq:boundary_conditionsk} in the sense of traces, equations \eqref{eq:ch1_k}, \eqref{eq:ch2_k}, \eqref{eq:u_k}, and \eqref{eq:z_k} hold a.e. in $\Omega$, and \cref{eq:nutrient_k} plus boundary condition \eqref{eq:sigma_nuk} hold in the weak sense
    \begin{equation*}
        \into \Dk\sigma \zeta \dx + \into \nabla \sigmak \cdot \nabla \zeta \dx + \alpha \intg (\sigmak - \sigmagk) \zeta\dA = \into  S_k \zeta \dx
    \end{equation*}
    for all $\zeta \in H^1$.
\end{prop}

\begin{proof}
    \noindent \textbf{\textit{Nutrient equation.}} 
First of all, we can rewrite the system 
\begin{equation}
    \begin{cases}
        \Dk\sigma - \Delta \sigmak = -\lambda_c\sigmak g(\phikm,\zkm) +\Lambdac(\zkm)(\sigmack-\sigmak) & \text{ in }\Omega\\
        \nabla\sigmak \cdot \nnu + \alpha(\sigmak -\sigmagk) = 0 & \text{ on }\Gamma
    \end{cases}
\end{equation}
in the more convenient form
\begin{equation} \label{eq:nutrient_kk}
    \begin{cases}
        - \Delta \sigmak + c_k \sigmak = d_k & \text{ in } \Omega\\
        \nabla\sigmak \cdot \nnu + \alpha(\sigmak -\sigmagk) = 0 & \text{ on }\Gamma,
    \end{cases}
\end{equation}
where 
\begin{equation}\label{eq:nutrient_kk_notation}
    c_k \coloneqq \frac{1}{\tau} + \lambda_c g(\phikm,\zkm) + \Lambdac(\zkm),\qquad d_k=\frac{\sigmakm}{\tau} + \Lambdac(\zkm) \sigmack 
\end{equation}
are known terms in $L^{\infty}$ and $L^2$ respectively, with $c_k \geq 0$ a.e. in $\Omega$. The variational formulation of the problem is the following:
\begin{equation*} 
    \begin{cases}
        \text{Find a } \sigma_k \in H^1 \text{ such that } \forall \zeta \in H^1\\
        \displaystyle \into \nabla \sigmak \cdot \nabla \zeta \dx + \alpha \intg \sigmak \zeta \dA +  \into c_k \sigmak \zeta \dx = \into d_k \zeta \dx +  \intg \sigmagk\zeta \dA.
    \end{cases}
\end{equation*}
Using Lax--Milgram theorem, one can show that there exists a unique weak solution $\sigmak \in H^1$.\\

\noindent \textbf{\textit{Cahn--Hilliard equation.}}
We consider the problem:
\begin{equation}\label{eq:ch_k}
    \begin{cases} 
    \Dk\varphi - \Delta \muk = U_k -\tau\Dk\mu & \text{in } \Omega\\
    \muk = -\Delta \phik + \psiu'(\phik) + \psin'(\phikm) + W_{,\varphi}(\phik, \epsilonukm, \zkm) + \tau \Dk \varphi  & \text{in } \Omega\\
    \nabla\phik \cdot \nnu = \nabla \muk \cdot \nnu = 0  & \text{on } \Gamma.
\end{cases}
\end{equation}  
The first equation in \eqref{eq:ch_k} can be reformulated in the equivalent form
\begin{equation} \label{eq:rewrite_chk}
    \frac{1}{\tau} (I-\Delta)^{-1} \phik + \muk = (I-\Delta)^{-1}\left(U_k + \mukm + \frac{1}{\tau}\phikm\right),
\end{equation}
observing that $I-\Delta : \mathcal{D}(-\Delta) \subseteq L^2 \to L^2$ (with Neumann homogeneous boundary condition) is a bijective operator, so $\gamma \coloneqq (I - \Delta)^{-1}: L^2 \to L^2$ is injective. Moreover, $-\Delta: \mathcal{D}(-\Delta) \subseteq L^2 \to L^2$ is a linear single-valued maximal monotone operator and, as a consequence, $\gamma$ is a linear, single-valued, monotone and contractive operator defined on all $L^2$. Substituting $\muk$ in the second equation of \eqref{eq:ch_k} and recalling the expression of $W_{,\varphi}$ from \eqref{eq:W_deriv_phi}, we obtain:
\begin{equation*}
    \begin{split}
        \frac{1}{\tau} & \gamma(\phik) -\Delta \phik + \psiu'(\phik) + \left(\hkm \C \R : \R + 1\right)\phik\\
        &= \gamma\left(U_k + \mukm + \frac{1}{\tau}\phikm\right) - \psin'(\phikm) + \phikm + \hkm \C \epsilonukm : \R.
\end{split}
\end{equation*}
For brevity, we introduce the known functions
\begin{align*}
        &j_k \coloneqq
        \gamma\left(U_k + \mukm + \frac{1}{\tau}\phikm\right) - \psin'(\phikm) + \phikm + \hkm \C \epsilonukm : \R, \\
        &l_k \coloneqq \hkm \C \R : \R + 1.
\end{align*}
We notice that $j_k \in L^2$ and that $l_k$ is bounded from above, since $h$ and $\C$ are bounded by Hypotheses \ref{hyp:elastic_viscous_coeff} and \ref{hyp:elastic_viscous_tensors} respectively, and satisfies $l_k \geq 1$, because $h$ is non-negative and $\C$ is strongly elliptic by Hypotheses \ref{hyp:elastic_viscous_coeff} and \ref{hyp:elastic_viscous_tensors}. To find a solution for 
\begin{equation}\label{eq:chk_redefined}
    \frac{1}{\tau} \gamma(\phik) -\Delta \phik + \psiu'(\phik) + l_k \phik = j_k,
\end{equation}
we introduce $\psiu_{\delta}$, the Moreau--Yosida approximation of $\psiu$ with regularisation parameter $\delta >0$.
We define the operator 
\begin{equation*}
    B_{\tau, k}^{\delta} \coloneqq \frac{1}{\tau} \gamma + \psiu_{\delta}' + l_k I : L^2 \to L^2.
\end{equation*}
We can reformulate the regularised system in the abstract form:
\begin{equation}\label{eq:chk_abstract}
        (B_{\tau, k}^{\delta}-\Delta)(\phid) = j_k. 
\end{equation}
The operator $-\Delta$ is maximal monotone. $B_{\tau,k}^{\delta}$ is monotone (because it is the sum of monotone operators) and hemicontinuous (because it is continuous). Finally, it is easy to show that $B_{\tau,k}^{\delta} - \Delta$ is coercive.  So, we can apply \cite[][Corollary 1.3, p. 48]{barbu1976}, and conclude that $B_{\tau,k}^{\delta} - \Delta$ is maximal monotone and that $\mathcal{R}(B_{\tau,k}^{\delta} - \Delta) = L^2$. This leads to the fact that it exists a 
\begin{equation*}
    \phid \in \mathcal{D}(B_{\tau,k}^{\delta} - \Delta) = L^2 \cap \mathcal{D}(-\Delta)= \{ v \in H^2 | \nabla v \cdot \nnu = 0 \text{ on } \Gamma \}
\end{equation*}
that satisfies \eqref{eq:chk_abstract}. Note that, obviously, $\phid$ also depends on $k$ and $\tau$, but at this level, they are fixed, so we omit this dependence to not overload the notation. Now it only remains to pass to the limit for $\delta \to 0$ and show that the limit satisfies \eqref{eq:chk_redefined}. We need some a priori estimates.
\begin{itemize}
    \item[] \textit{First a priori estimate.} We test \eqref{eq:chk_abstract} with $\phid$:
    \begin{equation*}
        \into \frac{1}{\tau} \gamma(\phid)\phid + |\nabla \phid|^2 + \psiu_{\delta}'(\phid)\phid + l_k \phid^2 \dx = \into j_k \, \phid  \dx.
    \end{equation*}
    Using the fact that $\gamma$ is monotone with $\gamma(0)=0$, that $\psiu_{\delta}'$ is monotone with $\psiu_{\delta}'(0)=0$ and $l_k \geq 1$, we have
    \begin{equation*}
        \norm{\phid}_{H^1}^2 \leq C \norm{\phid}_{L^2} \norm{j_k}_{L^2},
    \end{equation*}
    from which we get $\norm{\phid}_{H^1} \leq C \norm{j_k}_{L^2}= C_{\tau}$, where $C_{\tau}$ does not depend on $\delta$. 
    \item[] \textit{Second a priori estimate.} We test \eqref{eq:chk_abstract} with $-\Delta \phid + \psiu_{\delta}'(\phid)$ and, since $l_k$ is uniformly bounded from above and $\gamma$ is a contraction, by the first a priori estimate we get:
    \begin{equation*}
        \norm{-\Delta \phid + \psiu_{\delta}'(\phid)}_{L^2} \leq \norm{-\frac{1}{\tau}\gamma(\phid) - l_k \phid + j_k}_{L^2} \leq C_{\tau}.
    \end{equation*}
    On the other hand, we have
    \begin{equation*}
        \begin{split}
                \norm{-\Delta \phid + \psiu_{\delta}'(\phid)}_{L^2}^2 &= \into |-\Delta \phid|^2 \dx + \into |\psiu_{\delta}'(\phid)|^2 \dx + 2 \into -\Delta \phid \, \psiu_{\delta}'(\phid) \dx\\
                &= \into |-\Delta \phid|^2 \dx + \into |\psiu_{\delta}'(\phid)|^2 \dx + 2 \into\psiu_{\delta}''(\phid) |\nabla \phid|^2 \dx\\
                &\geq \norm{-\Delta \phid}_{L^2}^2 + \norm{\psiu_{\delta}'(\phid)}_{L^2}^2
        \end{split}
    \end{equation*}
    because, recalling that $\psiu_{\delta}'$ is Lipschitz and non-decreasing, $\psiu_{\delta}'' \geq 0$ a.e.
\end{itemize}
From the first and the second a priori estimates, we get that $\norm{\phid}_{H^2}+\norm{\psiu_{\delta}'(\phid)}_{L^2} \leq C_{\tau}$, so there exist a $\phik \in H^2$ and a $\rho_{\tau}^k \in L^2$ such that, along a non-relabelled subsequence, $\phid \rightharpoonup \phik$ in $H^2$, $\phid \to \phik$ in $L^2$ and $\psiu_{\delta}'(\phid) \rightharpoonup \rho_{\tau}^k$ in $L^2$. Furthermore, because of these convergences, 
\begin{equation*}
    \lim_{\delta \to 0} \into \psiu'_{\delta}(\phid) \, \phid \dx = \into \rho_{\tau}^k \, \phik \dx.
\end{equation*}
So, thanks to \cite[][Proposition 1.1, p. 42]{barbu1976}, we have that $\rho_{\tau}^k = \psiu'(\phik)$. Pointing out that $\gamma(\phid) \to \gamma(\phik)$ in $L^2$ since $\gamma$ is a contraction, we can pass to the weak limit in \eqref{eq:chk_redefined} and deduce that
\begin{equation*}
    \gamma(\phik) - \Delta \phik + \psiu'(\phik) + l_k \phik = j_k
\end{equation*}
in $L^2$. Additionally, we remark that $\nabla \phik \cdot \nnu = 0$ on $\Gamma$ in the sense of traces because $\nabla \phid \cdot \nnu= 0$ for every $\delta$ and the normal trace operator is linear and continuous over $H^2$.\\ 
Finally, we define $\muk$ as in the second equation of the system \eqref{eq:ch_k} and claim that it belongs to
\begin{equation*}
    \mathcal{R}((I-\Delta)^{-1}) = \mathcal{D}(I-\Delta) = \{ v \in H^2 | \nabla v \cdot \nnu = 0 \text{ on } \Gamma \}
\end{equation*}
by comparison in \eqref{eq:rewrite_chk}.
It remains to prove that the solution $(\phik,\muk)$ is unique. We take two solutions and the components $\varphi_1$ and $\varphi_2$ solving \eqref{eq:chk_redefined}. They satisfy
\begin{equation*}
    (\gamma(\varphi_1) - \gamma(\varphi_2)) - \Delta (\varphi_1 -\varphi_2) + \psiu'(\varphi_1) - \psiu'(\varphi_2) + l_k(\varphi_1-\varphi_2) = 0
\end{equation*}
in $L^2$. Testing this equation with $\varphi_1 - \varphi_2$, we have
\begin{multline*}
    \into (\gamma(\varphi_1) -\gamma(\varphi_2) )(\varphi_1 - \varphi_2) \dx + \into |\nabla(\varphi_1 - \varphi_2)|^2 \dx\\ + \into (\psiu'(\varphi_1) - \psiu'(\varphi_2))(\varphi_1 - \varphi_2) \dx
        + \into l_k (\varphi_1 - \varphi_2)^2 \dx  = 0.
\end{multline*}
Since $\gamma$ and $\psiu'$ are monotone, the first and third addends are non-negative. Moreover, given that $l_k \geq 1$, we get
\begin{equation*}
     \into (\varphi_1 - \varphi_2)^2 \dx + \into |\nabla (\varphi_1 - \varphi_2) |^2 \dx \leq 0,
\end{equation*}
from which $\varphi_1 = \varphi_2$ follows. Consequentially, also the components $\mu_1$ and $\mu_2$ must coincide from \eqref{eq:rewrite_chk}.\\

\noindent \textbf{\textit{Damage differential equation.}}
We want to find a weak solution of
\begin{equation} \label{eq:damage_k}
    \begin{cases}
        \begin{split}
            &\Dk z - \Delta_p \zk + \beta_{\tau}(\zk) + \pi(\zkm)  \\
         &\mathcolor{white}{\Dk z } + \displaystyle\frac{\hu'(\zk) + \hn'(\zkm)}{2}\C[\epsilonukm-\R\phik] : [\epsilonukm-\R\phik] = 0
        \end{split} & \text{ in } \Omega\\
        (|\nabla \zk|^{p-2}\nabla \zk) \cdot \nnu = 0 & \text{ on }\Gamma
    \end{cases}
\end{equation}
using a minimizing procedure. So we introduce the functional $ \mathcal{F}_{\tau,k} : W^{1,p} \to \RR$ defined as follows:
\begin{equation*}
    \begin{split}
           \mathcal{F}_{\tau,k}(z) \coloneqq &\frac{1}{2\tau} \int_{\Omega} |z|^2 \dx - \frac{1}{\tau} \int_{\Omega} \zkm z \dx + \frac{1}{p} \int_{\Omega} |\nabla z|^p \dx + \int_{\Omega} \widehat{\beta}_{\tau} (z) \dx  \\&+  \int_{\Omega} \pi(\zkm)z \dx + \int_{\Omega} \frac{\hu(z)}{2} \C[\epsilonukm-\R\phik] : [\epsilonukm-\R\phik] \dx \\&+ \int_{\Omega} \frac{\hn'(\zkm)}{2} \C[\epsilonukm-\R\phik] : [\epsilonukm-\R\phik] z \dx
    \end{split}
\end{equation*}
and we use the direct method of the Calculus of Variations. We consider a minimizing sequence $\{z_j\}_j$ and prove that it admits a subsequence that converges to a minimizer for $\mathcal{F}_{\tau,k}$. We will need coercivity and weakly lower semicontinuity of $\mathcal{F}_{\tau,k}$.
\begin{itemize}
    \item [] \textit{Coercivity.} Reminding that $\widehat{\beta}_{\tau}$ and $\hu$ are nonnegative, $\pi$ and $\hn'$ are Lipschitz, $\zkm \in W^{1,p} \hookrightarrow L^\infty$ since $p$ is strictly bigger than $d$, and $\C$ is bounded and strongly elliptic, we obtain:
    \begin{equation*}
        \mathcal{F}_{\tau,k}(z) \geq \frac{1}{2\tau} \int_{\Omega} |z|^2 \dx - C \int_{\Omega} |z| \dx + \frac{1}{p} \int_{\Omega} |\nabla z|^p \dx  -C\int_{\Omega}  |\epsilonukm-\R\phik|^2 |z| \dx . 
    \end{equation*}
    Using the Young inequality and $\epsilonukm-\R\phik \in H^1 \hookrightarrow L^4$, the previous inequality becomes:
    \begin{equation*}
        \mathcal{F}_{\tau,k}(z) \geq C \int_{\Omega} |z|^2 \dx + C \int_{\Omega} |\nabla z|^p \dx  - C.
    \end{equation*}
    \item [] \textit{Weakly lower semicontinuity.} All terms are convex and continuous in the strong topology and, therefore, weakly lower semicontinuous (see \cite[][Corollary 3.9, p. 61]{brezis2011}).  
\end{itemize}
We note that it exists $C \in \RR$ such that $\inf_{W^{1,p}} \mathcal{F}_{\tau,k} < C$, so we can suppose without loss of generality that $\mathcal{F}_{\tau,k}(z_j) < C$ for every $j$. Thanks to coercivity, it trivially follows that $\{ z_j\}_j$ is bounded in $W^{1,p}$. Thus, there exists a subsequence that we do not relabel and a $\zk \in W^{1,p}$ such that $z_j \rightharpoonup \zk$ in $W^{1,p}$. From weakly lower semicontinuity, we get that:
\begin{equation*}
    \mathcal{F}_{\tau,k}(\zk) \leq \liminf_{j \to +\infty} \mathcal{F}_{\tau,k}(z_j) = \inf_{W^{1,p}}\mathcal{F}_{\tau,k},
\end{equation*}
so $\zk$ is a minimizer for $\mathcal{F}_{\tau,k}$. To conclude, we observe that $\mathcal{F}_{\tau,k}$ is Fréchet differentiable, so the minimum $\zk$ satisfies the following associated Euler-Lagrange equation 
\begin{equation} \label{eq:damage_eulerlagrange}
    \begin{split}
        0 = &\frac{1}{\tau} \int_{\Omega} \zk w \dx - \frac{1}{\tau} \int_{\Omega} \zkm w \dx\\
        &\quad + \int_{\Omega} |\nabla \zk|^{p-2}\nabla \zk \cdot \nabla w \dx + \int_{\Omega} \beta_{\tau} (\zk) w \dx
        + \int_{\Omega} \pi(\zkm)w \dx\\ 
        &\quad + \int_{\Omega} \frac{\hu'(\zk)+\hn'(\zkm)}{2} \C[\epsilonukm-\R\phik] : [\epsilonukm-\R\phik] w\dx 
    \end{split}
\end{equation}
for every $w \in W^{1,p}$. By comparison in  \eqref{eq:damage_eulerlagrange}, since
\begin{equation*}
     \frac{\zk-\zkm}{\tau} + \beta_{\tau} (z) + \pi(\zkm) + \frac{\hu'(\zk)+\hn'(\zkm)}{2} \C[\epsilonukm-\R\phik] : [\epsilonukm-\R\phik] 
\end{equation*}
belongs to $L^2$,  it follows that also $-\Delta_p \zk = - \diver(|\nabla \zk|^{p-2}\nabla \zk)$ belongs to $L^2$. This means that $\zk \in \mathcal{D}(-\Delta_p)$.
Now we prove that the solution is unique. If we suppose to have two solutions to \eqref{eq:damage_k} $z_1$ and $z_2$, they both are minimizers of $\mathcal{F}_{\tau,k}$ and satisfy \eqref{eq:damage_eulerlagrange}. If we consider the difference between the two equations and we take $w=z_1-z_2$ as test function, we obtain:
\begin{equation*}
    \begin{split}
        0 = &\frac{1}{\tau} \int_{\Omega} (z_1-z_2)^2 \dx  + \int_{\Omega} (|\nabla z_1|^{p-2}\nabla z_1 - |\nabla z_2|^{p-2}\nabla z_2) \cdot \nabla (z_1-z_2) \dx \\
        &\quad + \int_{\Omega} \frac{\hu'(z_1)-\hu'(z_2)}{2}(z_1-z_2) \C[\epsilonukm-\R\phik] : [\epsilonukm-\R\phik] \dx \\
        &\quad + \int_{\Omega} \left(\beta_{\tau} (z_1) - \beta_{\tau} (z_2)\right) (z_1-z_2) \dx  \geq \frac{1}{\tau} \norm{z_1-z_2}_{L^2}^2,
    \end{split}
\end{equation*}
where the last inequality follows from the fact that $-\Delta_p$, $\beta_{\tau}$, and $\hu'$ are monotone operators, so the related terms are non-negative. Thus, it turns out that $z_1 = z_2$.\\

\noindent \textbf{\textit{Displacement equation.}} First of all, we rewrite the system
\begin{equation}
    \begin{cases}
         \Dk^2\uu - \diver\left[ a(\zk)\V \epsilonudk + h(\zk) \C (\epsilonuk - \R\phik)   \right] = \mathbf{0} & \text{ in } \Omega\\
         \uk = \mathbf{0} & \text{ on } \Gamma
    \end{cases}
\end{equation}
as 
\begin{equation} \label{eq:displacement_k}
    \begin{cases}
         - \diver\left[ \T_{k} \epsilonuk  \right] + \uk = \mathbf{t}_k & \text{ in } \Omega\\
         \uk = \mathbf{0} & \text{ on } \Gamma,
    \end{cases}
\end{equation}
where we have introduced the following known terms:
\begin{align}
    &\T_k \coloneqq \tau a(\zk) \V + \tau^2  h(\zk) \C = (\tau \omega a(\zk) + \tau^2  h(\zk))\C = \theta(\zk) \C,\\
    &\mathbf{t}_k \coloneqq 2 \ukm - \ukmm -\diver\left[ \tau^2h(\zk)\phik\C\R +  \tau a(\zk)\V\epsilonukm \right].
\end{align}
Since $\T_k$ is bounded and coercive and $\mathbf{t}_k$ is in $L^2$, it is easy to prove using Lax--Milgram theorem that system \eqref{eq:displacement_k}
has a (unique) weak solution $\uk \in H^1_0$. It remains to be proved that $\uk \in H^2$, and it can be done exactly as in \cite[][Lemma 4.1, p. 4596]{Heinemann_Rocca_2015}, using a bootstrap argument. Notice that here is where we need to require $\V = \omega \C$.\\  
\end{proof}

\noindent Given a sequence of scalar or vector-valued functions $\{w_{\tau}^k\}_{k=0}^{K_{\tau}}$ defined over $\Omega$, we introduce the piecewise constant interpolations $\overline{w}_{\tau}, \underline{w}_{\tau}$ and the piecewise linear interpolation $w_{\tau}$ over the time interval $[0,T]$ as
\begin{equation}
    \overline{w}_{\tau}(t)\coloneqq w_{\tau}^k,\qquad \underline{w}_{\tau}(t) \coloneqq w_{\tau}^{k-1}, \qquad
    w_{\tau}(t) \coloneqq \frac{t - t_{\tau}^{k-1}}{\tau} w_{\tau}^k + \frac{t_{\tau}^k -t}{\tau} w_{\tau}^{k-1}
\end{equation}
for every $t \in I_{\tau}^k$. 
With this new notation, the time-discretised and regularised system \eqref{eq:problemK} can be written as 
\begin{subequations} \label{eq:system_tau}
    \begin{align}
        &\partial_t\phit -\Delta \muto =  
        \bigg(\frac{\lambda_p \sigmato}{1+|W_{,\varepsilon}(\phitu,\epsilonutu,\ztu)|}-\lambda_a + \fto\bigg)g(\phitu,\ztu)
        -(\muto-\mutu), \label{eq:ch1_tau}\\
        &\muto = -\Delta\phito + \psiu'(\phito) + \psin'(\phitu)
            - h(\ztu)(\epsilonutu -\R \phito):\C\R + (\phito-\phitu), \label{eq:ch2_tau}\\
        &\partial_t \sigmat -\Delta\sigmato = -\lambda_c\sigmato g(\phitu,\ztu) + \Lambdac(\ztu)(\sigmacto-\sigmato),\label{eq:sigma_tau}\\
        &\partial_t \vt - \diver\left[ a(\zto)\V\epsilonvto + h(\zto)\C(\epsilonuto -\R\phito)\right] = \mathbf{0},\label{eq:u_tau}\\
        &\begin{aligned}
            \partial_t\zt -\Delta_p \zto + \beta_{\tau}(\zto) &+ \pi(\ztu)\\
            &+ \frac{\hu'(\zto)+\hn'(\ztu)}{2}(\epsilonutu-\R\phito):\C(\epsilonutu-\R\phito) = 0.
        \end{aligned} \label{eq:z_tau}
    \end{align}
\end{subequations}

\subsection{A priori estimates for the time-discrete system}
In the following, we will need the boundedness of the nutrient variable $\sigmak$, so we prove a comparison principle.
\begin{lemma} \label{sigma_bounded}
    The function $\sigmak$ satisfies $0 \leq \sigmak \leq M$ for every $k = 0, \dots, K_{\tau}$.
\end{lemma}
\begin{proof}
    Knowing that $\sigma_{\tau}^0=\sigma_0$ satisfies this property by Hypothesis \ref{hyp:initial_conditions}, we proceed by induction on $k$, so we suppose that $0 \leq \sigmakm \leq M$ and we prove that the same stands for $\sigmak$. We remind that $0 \leq \sigmack,\ \sigmagk \leq M$ and that, using the notation introduced in \eqref{eq:nutrient_kk_notation}, $c_k\geq 1/\tau$ and $d_k \geq 0$. We also recall that, given a function $F$, its positive and negative parts are defined as
    \begin{equation*}
        F_+(x) \coloneqq \max\{F(x),0\},\qquad F_-(x) \coloneqq \max\{-F(x),0\},
    \end{equation*}
    and that, if $F \in H^1$, the following relations hold
    \begin{equation}\label{eq:property_pn_part}
        \begin{split}
            \into  F\, F_{+} \dx =  \norm{ F_{+}}_{L^2}^2, &\qquad  \into  F\,  F_{-} \dx = - \norm{F_{-}}_{L^2}^2, \\
            \into \nabla F \cdot \nabla F_{+} \dx =  \norm{\nabla F_{+}}_{L^2}^2,
          & \qquad \into \nabla F \cdot \nabla F_{-} \dx =  - \norm{\nabla F_{-}}_{L^2}^2.
        \end{split} 
    \end{equation} 
    Testing \eqref{eq:nutrient_k} with $-(\sigmak)_-$, we obtain
    \begin{multline*}
        -\into \nabla \sigmak \cdot \nabla [(\sigmak)_-] \dx - \intg \alpha (\sigmak - \sigmagk) (\sigmak)_- \dA - \into c_k \sigmak (\sigmak)_- \dx\\
        = -\into d_k (\sigmak)_- \dx.
    \end{multline*}
    Using \eqref{eq:property_pn_part}, it holds
    \begin{equation*}
        \begin{split}
            \frac{1}{\tau} \norm{(\sigmak)_-}_{L^2}^2 
            &\leq \norm{\nabla [(\sigmak)_-]}_{L^2}^2 + \norm{\sqrt{c_k}(\sigmak)_-}_{L^2}^2 + \norm{\sqrt{\alpha}(\sigmak)_-}_{L_{\Gamma}^2}^2\\ 
            &= - \into d_k (\sigmak)_- \dx -\intg \alpha\, \sigmagk (\sigmak)_- \dA \leq 0,
        \end{split}
    \end{equation*}
    so $\norm{(\sigmak)_-}_{L^2}^2 = 0$ (or, equivalently, $\sigmak \geq 0$ a.e. in $\Omega$).\\
    In the same way, we test \eqref{eq:nutrient_k} with $(\sigmak - M)_+$, obtaining
    \begin{equation*}
        \begin{split}
                \into \nabla \sigmak \cdot \nabla [(\sigmak - M)_+] \dx \, + \intg \alpha (\sigmak - \sigmagk) (\sigmak - M)_+ \dA&\\
                + \into c_k \sigmak (\sigmak - M)_+ \dx &= \into d_k (\sigmak - M)_+ \dx,
        \end{split}
    \end{equation*}
    that can be rewritten as
    \begin{equation*}
        \begin{split}
            \into |\nabla [(\sigmak - M)_+]|^2 \dx  &+ \intg \alpha [(\sigmak - M)_+]^2 \dA - \intg \alpha (\sigmagk - M) (\sigmak - M)_+ \dA\\
            &+ \into c_k [(\sigmak - M)_+]^2 \dx + 
              \into (c_k M - d_k)(\sigmak - M)_+ \dx = 0. 
        \end{split}
    \end{equation*}
    Noticing that
    \begin{equation*}
        c_k M - d_k = \frac{1}{\tau}(M-\sigmakm) + \Lambdac(\zkm)(M-\sigmack) + \lambda_c M g(\phikm,\zkm) \geq 0,
    \end{equation*}
   and recalling that $c_k \geq 1/\tau$, from the previous inequality it follows that $\norm{(\sigmak - M)_+}_{L^2}^2$ is equal to $0$, so $\sigmak \leq M$ a.e. in $\Omega$.\\
\end{proof}

\begin{remark}
From \Cref{sigma_bounded} and \eqref{eq:estimate_dataK} it follows that:
    \begin{itemize}
        \item $U_k = \displaystyle \bigg(\frac{\lambda_p\sigmak}{1+|W_{,\varepsilon}(\phikm,\epsilonukm,\zkm)|} -\lambda_a + \fk\bigg)g(\phikm,\zkm)$ belongs to $L^2$ with $\norm{U_k}_{L^2} \leq C$ for a positive $C$ independent of $\tau$ and $k$,
        \item $S_k = -\lambda_c \sigmak g(\phikm,\zkm) + \Lambdac(\zkm) (\sigmack - \sigmak)$ is in $L^{\infty}$ with $\norm{S_k}_{L^{\infty}} \leq C$ for a positive $C$ independent of $\tau$ and $k$.
    \end{itemize}
\end{remark}

\begin{prop} \label{prop:a_priori_estimate}
    The time-discrete solution to the problem \eqref{eq:system_tau}  constructed from \Cref{prop:discrete_existence} satisfies the following a priori estimates uniformly in $\tau:$
    \begingroup
    \allowdisplaybreaks
    \begin{align}
        \norm{\phito}_{L^{\infty}(H^1) \cap L^{2}(H^2)} + \norm{\phitu}_{L^{\infty}(H^1)} \leq C,\label{eq:estimate_phi}\\
        \tau^{-1/2} \norm{\phito - \phitu}_{L^2(L^2)}\leq C,\label{eq:estimate_phi_jump}\\
        \norm{\partial_t\varphi_{\tau}}_{L^2((H^1)')} \leq C,\label{eq:estimate_phi_derivative}\\
        \norm{\psiu'(\phito)}_{L^2(L^2)} + \norm{\psiu'(\phitu)}_{L^2(L^2)} \leq C,\label{eq:estimate_psi_derivative}\\
        \norm{\muto}_{L^2(H^1)} + \norm{\mutu}_{L^2(H^1)} \leq C,\label{eq:estimate_mu}\\
        \norm{\sigmato}_{L^{\infty}(L^2)\cap L^2(H^1)} + \norm{\sigmatu}_{L^{\infty}(L^2)\cap L^2(H^1)} \leq C,\label{eq:estimate_sigma}\\
        \norm{\partial_t\sigma_{\tau}}_{L^2((H^1)')} \leq C,\label{eq:estimate_sigma_derivative}\\
        \norm{\uto}_{L^{\infty}(H^2)} + \norm{\utu}_{L^{\infty}(H^2)} \leq C,\label{eq:estimate_u}\\
        \norm{\ut}_{W^{1,\infty}(H^1) \cap H^1(H^2)} \leq C,\label{eq:estimate_li_u}\\
        \norm{\vto}_{L^{\infty}(H^1) \cap L^2(H^2)} + \norm{\vtu}_{L^{\infty}(H^1)} \leq C,\label{eq:estimate_v}\\
        \norm{\vt}_{L^{\infty}(H^1)\cap H^1(L^2)} \leq C,\label{eq:estimate_li_v}\\
        \norm{\zto}_{L^{\infty}(W^{1,p}) \cap L^2(W^{1+\delta,p})} + \norm{\ztu}_{L^{\infty}(W^{1,p}) \cap L^2(W^{1+\delta,p})}  \leq C,\label{eq:estimate_z}\\
        \norm{z_{\tau}}_{L^{\infty}(W^{1,p})  \cap L^2(W^{1+\delta,p}) \cap H^1(L^2)} \leq C,\label{eq:estimate_li_z}\\
        \norm{-\Delta_p \zto}_{L^2(L^2)} + \norm{-\Delta_p \ztu}_{L^2(L^2)}  \leq C,\label{eq:estimate_delta_z}\\
        \norm{\beta_{\tau}(\zto)}_{L^2(L^2)}  \leq C\label{eq:estimate_xi},
    \end{align}
    \endgroup
    where $\delta \in (0,1/p)$.
\end{prop}
\noindent Notice that in \cref{eq:estimate_phi,eq:estimate_v} the estimates for the retarded piecewise constant interpolants hold in weaker spaces because they are equal to the initial data in $[0,\tau]$ and the initial data are less regular than the corresponding discrete solutions at the step $k=1,\dots,K_{\tau}$.
\begin{proof}
    \textbf{\textit{Energy estimate.}} Testing \eqref{eq:ch1_k} with $\tau \muk$, we obtain:
    \begin{equation*}
        \tau \into \Dk\varphi \, \muk \dx + \tau \into |\nabla \muk|^2 \dx = \tau \into U_k\, \muk \dx - \tau \into (\muk -\mukm) \muk \dx.
    \end{equation*}
    Using the Young inequality to handle the last term, we have:
    \begin{equation}  \label{eq:energy_estimate_ch1}
            \begin{split}
                &\tau \into \Dk\varphi \, \muk \dx + \tau \into |\nabla \muk|^2 \dx + \frac{\tau}{2} \into |\muk|^2 \dx - \frac{\tau}{2} \into |\mukm|^2 \dx\\
                &\quad \leq \tau \into U_k\, \muk \dx.
            \end{split}
    \end{equation}
    Testing \eqref{eq:ch2_k} with $-(\phik-\phikm)$,
    \begin{equation*}
        \begin{split}
             -\tau &\into \muk\, \Dk \varphi \dx + \!\!\into \nabla \phik \cdot \nabla (\phik - \phikm) \dx + \!\!\into \left[\psiu'(\phik) + \psin'(\phikm)\right]\!(\phik-\phikm) \dx \\
            &+ \into W_{,\varphi}(\phik,\epsilonukm,\zkm)(\phik-\phikm) \dx + \into |\phik - \phikm|^2 \dx = 0.  
        \end{split}
    \end{equation*}
    Employing the Young inequality for the second term and \Cref{convex_concave_decomposition} for $\Psi = \psiu + \psin$, we get
    \begin{equation}  \label{eq:energy_estimate_ch2}
        \begin{split}
             -\tau &\into \muk\, \Dk \varphi \dx + \frac{1}{2}\into |\nabla \phik|^2 \dx - \frac{1}{2}\into |\nabla \phikm|^2 \dx\\
             &+ \into \Psi(\phik) \dx
             - \into \Psi(\phikm) \dx\\
             &+ \into W_{,\varphi}(\phik,\epsilonukm,\zkm)(\phik-\phikm) \dx + \into |\phik - \phikm|^2 \dx\leq 0.  
        \end{split}
    \end{equation}
    Testing \eqref{eq:nutrient_k} with $\tau \sigmak$ and applying the Young inequality for the first term, we obtain: 
    \begin{equation}  \label{eq:energy_estimate_sigma}
        \begin{split}
                \frac{1}{2} &\into  |\sigmak|^2 \dx - \frac{1}{2} \into  |\sigmakm|^2 \dx + \tau \into |\nabla \sigmak|^2 \dx + \tau \intg \alpha |\sigmak|^2 \dA\\
                &\leq  \tau \into S_k\,\sigmak \dx + \tau \intg \alpha \sigmagk\,\sigmak \dA.
        \end{split}
    \end{equation}
    Testing \eqref{eq:u_k} with $\uk-\ukm = \tau \vk$, we get:
    \begin{equation*}
        \begin{split}
            \into (\vk - \vkm) \cdot \vk \dx\, &+\, \tau \into a(\zk) \V \epsilonvk : \epsilonvk \dx\\ &+ \into W_{,\varepsilon}(\phik,\epsilonuk,\zk) : (\epsilonuk - \epsilonukm) \dx = 0. 
        \end{split}
    \end{equation*}
    Exploiting the Young inequality for the first term, the fact that $a_* \leq a$ and that $\V$ is uniformly elliptic for the second term, we have:
        \begin{equation}  \label{eq:energy_estimate_u}
        \begin{split}
            \frac{1}{2}\into |\vk|^2 \dx &- \frac{1}{2}\into |\vkm|^2 \dx + C \tau \into |\epsilonvk|^2 \dx\\ &+ \into W_{,\varepsilon}(\phik,\epsilonuk,\zk) : (\epsilonuk - \epsilonukm) \dx \leq  0. 
        \end{split}
    \end{equation}
    Finally, we test \eqref{eq:z_k} with $\zk -\zkm$, obtaining:
    \begin{equation*}
        \begin{split}
                \tau \into& |\Dk z|^2 \dx + \into |\nabla \zk|^{p-2} \nabla \zk \cdot \nabla (\zk-\zkm) \dx \\&+ \into \beta_{\tau}(\zk)(\zk -\zkm) \dx + \into \pi(\zkm)(\zk-\zkm) \dx \\&+ \into  \Big[\Wzu(\phik,\epsilonukm, \zk)+\Wzn(\phik,\epsilonukm, \, \zkm)\Big](\zk-\zkm) \dx  = 0.
        \end{split}
    \end{equation*}
    Employing the Young for the second term, the convexity of $\widehat{\beta}_{\tau}$ for the third and moving the term with $\pi$ to the right-hand side, we get:
    \begin{equation} \label{eq:energy_estimate_z}
      \begin{split}
            \tau \into&  |\Dk z|^2 \dx + \frac{1}{p}\into |\nabla \zk|^p \dx - \frac{1}{p}\into |\nabla \zkm|^p \dx \\
            &+ \into \widehat{\beta}_{\tau}(\zk) \dx- \into \widehat{\beta}_{\tau}(\zkm) \dx \\
            &+ \into  \Big[\Wzu(\phik,\epsilonukm, \zk)+\Wzn(\phik,\epsilonukm, \, \zkm)\Big](\zk-\zkm) \dx  \\
            \leq& - \into \pi(\zkm)(\zk-\zkm) \dx =  -\tau \into \pi(\zkm)\Dk z \dx.
        \end{split}
    \end{equation}
   Now we notice that, since $W$ is convex with respect to its first and second variables, and since we can apply 
    \Cref{convex_concave_decomposition} to $W = \oset{$\smallsmile$}{W}_{3}{} + \oset{$\smallfrown$}{W}_{3}{}$, we have:
    \begin{equation*}
       \begin{split}
            &W_{,\varphi}(\phik,\epsilonukm, \zkm)(\phik-\phikm) \geq W(\phik,\epsilonukm,\zkm)- W(\phikm,\epsilonukm,\zkm),\\
            &W_{,\varepsilon}(\phik,\epsilonuk, \zkm):(\epsilonuk-\epsilonukm) \geq W(\phik,\epsilonuk,\zk)- W(\phik,\epsilonukm,\zk),\\
            &\begin{split}
                \Big[\Wzu(\phik,\epsilonukm, \zk)+\Wzn(\phik,\epsilonukm,& \, \zkm)\Big](\zk-\zkm) \\
            \geq& \,W(\phik,\epsilonukm,\zk)- W(\phik,\epsilonukm,\zkm).
            \end{split}
       \end{split}
    \end{equation*}
    So, by summing the three above inequalities, we obtain that the left-hand side is greater or equal than
    \begin{equation*}
        W(\phik,\epsilonuk,\zk)- W(\phikm,\epsilonukm,\zkm).
    \end{equation*}
    Adding \eqref{eq:energy_estimate_ch1}, \eqref{eq:energy_estimate_ch2}, \eqref{eq:energy_estimate_sigma}, \eqref{eq:energy_estimate_u}, \eqref{eq:energy_estimate_z} and employing the previous inequality regarding $W$, we infer that:
     \begingroup
    \allowdisplaybreaks
    \begin{equation} \label{eq:partial_energy_inequality}
         \begin{split}
            &\begin{split}
                \frac{\tau}{2}& \into |\muk|^2 \dx - \frac{\tau}{2} \into |\mukm|^2 \dx 
                + \frac{1}{2}\into |\nabla \phik|^2 \dx - \frac{1}{2}\into |\nabla 
                \phikm|^2 \dx\\ 
                & + \into \Psi(\phik) \dx  - \into \Psi(\phikm) \dx 
                + \frac{1}{2} \into  |\sigmak|^2 \dx - \frac{1}{2} \into  |\sigmakm|^2 \dx \\
                & + \frac{1}{2}\into |\vk|^2 \dx\, - \frac{1}{2}\into |\vkm|^2 \dx\,+ \frac{1}{p}\into |\nabla \zk|^p \dx - \frac{1}{p}\into |\nabla \zkm|^p \dx\\
                & + \into \widehat{\beta}_{\tau}(\zk) \dx  - \into \widehat{\beta}_{\tau}(\zkm) \dx+ \into W(\phik,\epsilonuk,\zk) \dx\\
                & - \into W(\phikm,\epsilonukm,\zkm) \dx+ \tau \bigg[\into |\nabla \muk|^2 \dx + \into \tau^{-1}|\phik-\phikm|^2 \dx\\
                &+ \into |\nabla \sigmak|^2 \dx + \intg \alpha |\sigmak|^2 \dA+ C \into |\epsilonvk|^2 \dx + \into |\Dk z|^2 \dx \bigg]
            \end{split}\\
            &\leq \tau \bigg[\into U_k\, \muk \dx + \into S_k\,\sigmak \dx + \intg \alpha \sigmagk\,\sigmak \dA - \into \pi(\zkm)\Dk z \dx \bigg] \\
            &\leq \tau \bigg[\into U_k\, \muk \dx + C + \into |\pi(\zkm)||\Dk z| \dx \bigg],
         \end{split}
    \end{equation}
    \endgroup
    where the latter inequality follows from the fact that $S_k$, $\sigmak$ and $\sigmagk$ are bounded in $L^{\infty}$ uniformly with respect to $k$ and $\tau$. Then, by the H\"older, Poincaré--Wirtinger, and Young inequalities, recalling that $\norm{U_k}_{L^2} \leq C$, we have
    \begin{equation} \label{eq:Uk_estimate}
        \begin{split}
            \into U_k\, \muk \dx &\leq \norm{U_k}_{L^2} \norm{\muk}_{L^2} \leq C \left(\norm{\muk - \langle\muk\rangle}_{L^2} + |\Omega|^{\frac{1}{2}}|\langle\muk\rangle|\right) \\
            &\leq C \left(\norm{\nabla \muk}_{L^2} + |\Omega|^{\frac{1}{2}}|\langle\muk\rangle|\right) 
            \leq \eta \into |\nabla \muk|^2 \dx + C_{\eta} + C|\langle\muk\rangle|,
        \end{split}
    \end{equation}
    where $\langle \muk \rangle$ denotes the mean value of $\muk$ and $\eta$ is a small positive constant yet to be defined. 
    Testing \eqref{eq:ch2_k} with $1$ and dividing by $|\Omega|$, we obtain 
    \begin{equation*}
        \langle \muk \rangle = \frac{1}{|\Omega|} \into \muk \dx  = \frac{1}{|\Omega|} \into \psiu'(\phik) + \psin'(\phikm) + W_{,\varphi}(\phik, \epsilonukm, \zkm) + \tau \Dk \varphi \dx.
    \end{equation*}
    Adding and subtracting $\psin'(\phik)$ and $\R\phikm$, employing growth assumption \eqref{eq:property_psi_derivative} of $\Psi$, the Lipschitz continuity of $\psin'$, and the boundedness of $h$, we have \begin{equation}\label{eq:mean_value_estimate}
        \begin{split}
            |\langle &\muk \rangle| \leq  \frac{1}{|\Omega|} \into \big|\psiu'(\phik) + \psin'(\phikm)\big| + Ch(\zkm)|\epsilonukm - \R \phik| + \tau |\Dk \varphi| \dx \\
            &\begin{split}
                \leq C\into \big|\Psi'(\phik)\big| + \big|\psin'(\phik)-\psin'(\phikm)\big| &+ h(\zkm)|\epsilonukm - \R \phikm| \\
                &+ h^*|\R\phikm -\R \phik| + \tau  |\Dk \varphi| \dx
            \end{split}\\
            &\leq C \into \Psi(\phik) + h(\zkm)|\epsilonukm - \R \phikm| + \tau |\Dk \varphi| \dx.
        \end{split}
    \end{equation}
    Using the Young inequality twice and $\C$
    strong ellipticity from Hypothesis \ref{hyp:elastic_viscous_tensors}, from the above inequality we obtain
    \begin{equation*} 
        \begin{split}
            |\langle \muk \rangle| &\leq  C \into \Psi(\phik) + h(\zkm)|\epsilonukm - \R \phikm|^2 \dx+ \tau \int \eta|\Dk \varphi|^2 + C_{\eta}\dx\\
            &\leq C \into \Psi(\phik) + W(\phikm,\epsilonukm,\zkm) \dx + \eta \into \tau^{-1} |\phik-\phikm|^2 \dx +  C_{\eta},
        \end{split}
    \end{equation*}
    where $\eta$ is the same small, positive constant introduced before.
    So, substituting in \eqref{eq:Uk_estimate}, we deduce that
    \begin{equation} \label{eq:Uk_final_estimate}
        \begin{split}
            \into U_k\, \muk \dx \leq & \,C \into \Psi(\phik) + W(\phikm,\epsilonukm,\zkm) \dx \\
            &+ \eta \into |\nabla \muk|^2 \dx + \eta \into \tau^{-1} |\phik-\phikm|^2 \dx +  C_{\eta}.
        \end{split}
    \end{equation}
    Moreover, recalling that, by Hypothesis \ref{hyp:pi}, $\pi$ is Lipschitz continuous, using the H\"older inequality and the Young inequality with a small $\eta$ yet to be defined, we get
    \begin{equation}\label{eq:pik_estimate}
        \begin{split}
            \into &|\pi(\zkm)||\Dk z| \dx \leq C \into (|\zkm| + 1)|\Dk z| \dx\\
            &\leq C (\norm{\zkm}_{L^2} + 1) \norm{\Dk z}_{L^2} \leq \eta \norm{\Dk z}_{L^2}^2 + C_{\eta} (\norm{\zkm}_{L^2}^2 + 1)\\ 
            &\leq \eta \norm{\Dk z}_{L^2}^2 + C_{\eta} \bigg(\sum_{i=1}^{k-1}\tau \norm{D_{\tau,i} z}_{L^2}^2 + 1\bigg),  
        \end{split}
    \end{equation}
    where we have also used the fact that $\zkm = z_0 + \sum_{i=1}^{k-1} \tau D_{\tau,i} z$ if $k \geq 2$ and  $\zkm = z_0$ if $k=1$.
    Finally, using inequalities \eqref{eq:Uk_final_estimate} and \eqref{eq:pik_estimate} in \eqref{eq:partial_energy_inequality}, moving to the left-hand side the terms with $\eta$ (fixing $\eta$ small enough) and summing from $k=1$ to $j$, we obtain
    \begingroup
    \allowdisplaybreaks
         \begin{align} \label{eq:energy_inequality}
            \frac{\tau}{2}& \into |\muj|^2 \dx  
            + \frac{1}{2}\into |\nabla \phij|^2 \dx + \into \Psi(\phij) \dx  + \frac{1}{2} \into  |\sigmaj|^2 \dx + \frac{1}{2}\into |\vj|^2 \dx \notag\\
            & + \frac{1}{p}\into |\nabla \zj|^p \dx
            + \into \widehat{\beta}_{\tau}(\zj) \dx   + \into W(\phij,\epsilonuj,\zj) \dx\notag\\ 
            &+ \sum_{k=1}^j\bigg[\tau \bigg(\frac{1}{2}\into |\nabla \muk|^2 \dx +\frac{1}{2} \into \tau^{-1}|\phik-\phikm|^2 \dx + \into |\nabla \sigmak|^2 \dx \\
            &+ \intg \alpha |\sigmak|^2 \dA + C \into |\epsilonvk|^2 \dx + \frac{1}{2}\into |\Dk z|^2 \dx \bigg)\bigg]\notag\\
            \leq & \, C_0 + C \sum_{k=1}^j \bigg[ \tau \bigg( \into \Psi(\phik) \dx  + \into W(\phikm,\epsilonukm,\zkm) \dx\notag\\
            &+ \sum_{i=1}^{k-1}\tau \into |D_{\tau,i} z|^2 \dx  \bigg)\bigg],\notag
        \end{align}
        \endgroup
       where $C$ does not depend on the initial data while
    \begin{equation*}
        \begin{split}
            C_0 = & \frac{1}{2}\into |\nabla \varphi_0|^2 \dx + \into \Psi(\varphi_0) \dx  + \frac{1}{2} \into  |\sigma_0|^2 \dx + \frac{1}{2}\into |\vv_0|^2 \dx + \frac{1}{p}\into |\nabla z_0|^p \dx \\
             &\quad + \into \widehat{\beta}_{\tau}(z_0) \dx  + \into W(\varphi_0,\varepsilon(\uu_0),z_0) \dx\\
             \leq \,  &C \bigg[ \norm{\varphi_0}_{H^1}^2 + \norm{\sigma_0}_{L^2}^2 + \norm{\vv_0}_{L^2}^2 + \norm{\uu_0}_{H^1}^2+\norm{z_0}_{W^{1,p}}^p + \into \Psi(\varphi_0) \dx + \into \widehat{\beta}(z_0) \dx\bigg].
        \end{split}
    \end{equation*}
    Here we used the fact that $\widehat{\beta}_{\tau}(z_0) \leq \widehat{\beta}(z_0)$ a.e. that comes directly from the definition of $\widehat{\beta}_{\tau}$ (see \cite[][Proposition 2.11, p. 39]{brezis1973}) and the following inequality regarding the elastic energy
    \begin{align*}
        \into W(\varphi_0,\varepsilon(\uu_0),z_0) \dx &= \into \frac{h(z_0)}{2} \C(\varepsilon(\uu_0) - \R \varphi_0):(\varepsilon(\uu_0) - \R \varphi_0) \dx \\
        & \leq C h^* \into |\varepsilon(\uu_0) - \R \varphi_0|^2 \dx \leq C (\norm{\uu_0}_{H^1}^2 + \norm{\varphi_0}^2).
    \end{align*}
    Applying the discrete Gronwall inequality stated in \Cref{lemma:discrete_gronwall} to \eqref{eq:energy_inequality} leads to the boundedness of the left-hand side, from which we have 
      \begin{equation}\label{eq:energy_estimate}
         \begin{split}
             &\norm{\phij}_{H^1} + \norm{\Psi(\phij)}_{L^1} + \norm{\sigmaj}_{L^2} + \norm{\vj}_{L^2}\\
             &\quad + \norm{\nabla \zj}_{L^p} + \norm{\widehat{\beta}(\zj)}_{L^2} + \norm{\epsilonuj}_{L^2} + \sum_{k=1}^{K_{\tau}} \bigg[\intk \bigg(\norm{\nabla \muk}_{L^2}^2\\
             &\quad+ \norm{\tau^{-1/2}(\phik-\phikm)}_{L^2}^2 + \norm{\nabla \sigmak}_{L^2}^2+ \norm{\epsilonvk}_{L^2}^2 + \norm{\Dk z}_{L^2}^2 \bigg) \ds \bigg]\leq C
         \end{split}
    \end{equation}
    and, as a consequence, \eqref{eq:estimate_phi_jump} and \eqref{eq:estimate_sigma}.\\

 \noindent \textbf{\textit{Energy estimate consequences.}} From the equality
    \begin{equation*}
        \zj = z_0 + \sum_{k=1}^j \intk \Dk z\ds ,
    \end{equation*}
   and \eqref{eq:energy_estimate}, we have $\norm{\zj}_{L^2} \leq C$.  By Poincaré--Wirtinger inequality,
   \begin{equation}\label{eq:estimate_zk_W1p}
       \norm{\zj}_{W^{1,p}} \leq C \left(\norm{\nabla \zj}_{L^p} + |\langle \zj \rangle|\right) \leq C.
   \end{equation}
   Here we used $\norm{\nabla \zj}_{L^p} \leq C$ by \eqref{eq:energy_estimate} and we controlled the mean value of $\zj$ with its bounded $L^2$ norm.
    We can also gain a mean value estimate for $\muk$. Combining the first line from \eqref{eq:mean_value_estimate} with \eqref{eq:energy_estimate}, we immediately obtain $|\langle \muk \rangle| \leq C$. As a consequence, exploiting Poincaré--Wirtinger inequality, it follows that
    \begin{equation*}
        \norm{\muj}_{L^2} \leq \norm{\muj - \langle \muj \rangle}_{L^2} + C|\langle \muj \rangle| \leq C \left(\norm{\nabla \muj}_{L^2} + 1 \right)
    \end{equation*}
    and, thanks to \eqref{eq:energy_estimate}, we get \eqref{eq:estimate_mu}. Notice that here we also employ $\mutu = 0$ in $[0,\tau]$ to obtain the estimate for $\mutu$. Finally, by comparison in \eqref{eq:ch1_tau} and \eqref{eq:sigma_tau}, we have \eqref{eq:estimate_phi_derivative} and \eqref{eq:estimate_sigma_derivative}.\\ 

    \noindent \textbf{\textit{Higher order estimate for the displacement.}} We test the equation \eqref{eq:u_k} with $-\tau \diver\left[\V\epsilonvk\right]$, obtaining
    \begin{align*}
        -&\into (\vk -\vkm) \cdot \diver\left[\V\epsilonvk\right] \dx + \tau \into \diver\left[h(\zk)\C\epsilonuk\right] \cdot\diver\left[\V\epsilonvk\right] \dx \\
            &\quad+ \tau \into \diver\Big[a(\zk)\V\epsilonvk\Big] \cdot \diver\left[\V\epsilonvk\right] \dx\\
            &= \tau \into \diver\left[h(\zk)\C\R\phik\right] \cdot\diver\left[\V\epsilonvk\right] \dx.
    \end{align*}
    Developing the obvious calculations in the second and third terms on the left-hand side and moving some terms to the right-hand side, we have
    \begin{equation*}
        \begin{split}
            &- \! \into \!  (\vk -\vkm) \! \cdot \! \diver\left[\V\epsilonvk\right] \! \! \dx + \tau  \! \into \!  a(\zk) \diver \Big[\V\epsilonvk\Big] \! \cdot \! \diver\left[\V\epsilonvk\right] \! \! \dx\\
            & \quad \!=  - \tau  \! \into \!  \left(h'(\zk)\C\epsilonuk \nabla\zk\right) \! \cdot \!\diver\left[\V\epsilonvk\right] \! \! \dx -\tau  \! \into \!  h(\zk)\diver\Big[\C\epsilonuk\Big] \! \cdot \!\diver\left[\V\epsilonvk\right] \! \! \dx\\
            & \quad \mathrel{\hphantom{=}} -  \tau  \! \into \!  \Big(a'(\zk)\V \epsilonvk \nabla \zk \Big) \! \cdot \! \diver\left[\V\epsilonvk\right] \! \! \dx + \tau  \! \into \!  \diver\left[h(\zk)\C\R\phik\right] \! \cdot \!\diver\left[\V\epsilonvk\right] \! \! \dx. 
        \end{split}
    \end{equation*}
    Concerning the first term on the left-hand side, recall that for every $k$ it holds $\uk=0$ on $\Gamma$ and, consequently, $\vk=0$ on $\Gamma$. Thus, it can be estimated as follows:
    \begin{equation}\label{eq:estimate_uk_I1}
        \begin{split}
            -&\into (\vk -\vkm) \cdot \diver\left[\V\epsilonvk\right] \dx = \into \left[\epsilonvk -\epsilonvkm\right] :\left[\V \epsilonvk\right] \dx\\
             &\geq \frac{1}{2} \into \epsilonvk:\V\epsilonvk \dx - \frac{1}{2} \into \epsilonvkm : \V\epsilonvkm \dx,
        \end{split}
    \end{equation}
    where the inequality holds because $\V$ is symmetric and positive-defined, so the associated quadratic form is convex.
    For the second left-hand term, since $a$ is bounded from below by a strictly positive constant by Hypothesis \ref{hyp:elastic_viscous_coeff}, using \Cref{lemma:H2_inequality} (and the fact that $\vk = 0$ on $\Gamma$), we obtain
    \begin{equation}\label{eq:estimate_uk_I5}
        \begin{split}
            \tau \into \Big(a(\zk) \diver \Big[\V\epsilonvk\Big]\Big) &\cdot \diver\left[\V\epsilonvk\right] \dx \\
            &\geq a_* \tau \into \Big|\diver\left[\V\epsilonvk\right]\Big|^2 \dx
            \geq C_* \tau \norm{\vk}_{H^2}^2.
        \end{split}
    \end{equation}
    The first term on the right-hand side can be estimated as follows:
    \begin{equation*}
        \begin{split}
           \bigg| - \tau \into \left(h'(\zk)\C\epsilonuk \nabla\zk\right) \cdot\diver\left[\V\epsilonvk\right] \dx \bigg| \leq C \tau \norm{\epsilonuk}_{L^q} \norm{\nabla \zk}_{L^p}\norm{\diver[\V\epsilonvk]}_{L^2},
        \end{split}
    \end{equation*}
    thanks to H\"older inequality. Here $q$ is chosen to satisfy $\frac{1}{q}+\frac{1}{p}+\frac{1}{2}=1$ and, since $p > d$ and $d=2$ or $d=3$, it is easy to check that $q \in [2,6)$. So, because of the embedding $H^1 \hookrightarrow L^q$, \Cref{lemma:H2_inequality}, the energy estimate \eqref{eq:energy_estimate} and the Young inequality, it follows:
    \begin{equation}\label{eq:estimate_uk_I2}
           \bigg| - \tau \into \left(h'(\zk)\C\epsilonuk \nabla\zk\right) \cdot\diver\left[\V\epsilonvk\right] \dx \bigg| \leq C_{\eta} \tau \norm{\uk}_{H^2}^2 + \eta \tau \norm{\vk}_{H^2}^2,
    \end{equation}
    where $\eta>0$ is small and yet to be chosen.\\
    Regarding the second term on the right-hand side, since $h$ is bounded, we deduce that
    \begin{equation}\label{eq:estimate_uk_I3}
        \begin{split}
            &\bigg|-\tau \into \Big(h(\zk)\diver\Big[\C\epsilonuk\Big]\Big) \cdot\diver\left[\V\epsilonvk\right] \dx \bigg|\\
            &\quad\leq C \tau \norm{\diver[\C\epsilonuk]}_{L^2} \norm{\diver[\V\epsilonvk]}_{L^2}\leq C\tau \norm{\uk}_{H^2}\norm{\vk}_{H^2}\\
            &\quad\leq \tau C_{\eta}\norm{\uk}_{H^2}^2+ \eta\tau\norm{\vk}_{H^2}^2.
        \end{split}
    \end{equation}
    We handle the third term on the right-hand side using the fact that $a$ is Lipschitz continuous by Hypothesis \ref{hyp:elastic_viscous_coeff}, the H\"older inequality, previous estimates, the Young inequality, the embedding $H^2 \hookrightarrow L^q$, and Ehrling's Lemma stated in \Cref{thm:ehrling_lemma}, obtaining
    \begin{equation}\label{eq:estimate_uk_I4}
        \begin{split}
            &\bigg|-  \tau \into \Big[a'(\zk)\V \epsilonvk \nabla \zk \Big] \cdot \diver\left[\V\epsilonvk\right] \dx \bigg|\\
            &\quad\leq C \tau \norm{\epsilonvk}_{L^q}\norm{\nabla \zk}_{L^p}\norm{\diver[\V\epsilonvk]}_{L^2}\\
            &\quad\leq C \tau \norm{\epsilonvk}_{L^q}\norm{\vk}_{H^2} \leq C_{\eta} \tau \norm{\epsilonvk}_{L^q}^2 +\eta\norm{\vk}_{H^2}^2\\
            &\quad\leq \eta \tau \norm{\vk}_{H^2}^2 + \left( \theta \tau \norm{\vk}_{H^2}^2+ C_{\eta,\theta} \tau \norm{\vk}_{L^2}^2\right) \leq (\eta + \theta)\tau\norm{\vk}_{H^2}^2 + C_{\eta,\theta}\tau
        \end{split}
    \end{equation}
    where $\eta, \theta >0$ are small and yet to be chosen.\\
    Finally, we turn our attention to the last term on the right-hand side. After noticing that
    \begin{equation*}
        \diver\left[h(\zk)\C\R\phik\right] = \C\R\left(h'(\zk)\phik \nabla \zk + h(\zk) \nabla \phik\right) + h(\zk)\phik \diver(\C\R),
    \end{equation*}
    we use the H\"older and the Young inequalities and the previous estimates. We get
    \begin{equation}\label{eq:estimate_uk_I6}
        \begin{split}
            &\bigg|\tau \into \diver\left[h(\zk)\C\R\phik\right] \cdot\diver\left[\V\epsilonvk\right] \dx\bigg| \\
            &\quad \leq C \tau \left(\norm{\phik}_{L^q} \norm{\nabla \zk}_{L^p} +\norm{\nabla \phik}_{L^2} +\norm{\phik}_{L^2}\right)\norm{\diver[\V\epsilonvk]}_{L^2}\\
            &\quad \leq C \tau \norm{\vk}_{H^2} \leq \eta \tau \norm{\vk}_{H^2}^2 + C_{\eta} \tau.
        \end{split}
    \end{equation}
    Putting together \eqref{eq:estimate_uk_I1}--\eqref{eq:estimate_uk_I6} and fixing $\eta$ and $\theta$ small enough lead to
    \begin{equation*}
        \begin{split}
            \frac{1}{2} \into \epsilonvk:\V \epsilonvk \dx - \frac{1}{2} \into \epsilonvkm:\V \epsilonvkm \dx + \frac{ C_* }{2}\tau \norm{\vk}_{H^2}^2 \leq C \tau (1+ \norm{\uk}_{H^2}^2).
        \end{split}
    \end{equation*}
    So, summing for $k=1$ to $j$ and recalling that $\V$ is coercive by Hypothesis \ref{hyp:elastic_viscous_tensors}, we get
    \begin{equation*}
        \begin{split}
            \norm{\epsilonvk&}_{L^2}^2 + \sum_{k=1}^j\tau \norm{\vk}_{H^2}^2 \leq C_0 + C\sum_{k=1}^j \tau \norm{\uk}_{H^2}^2 \\
            &\leq C_0 + C\sum_{k=1}^j \tau \Bigg[\sum_{l=1}^k \tau \norm{\vl}_{H^2}^2+ \norm{\uu_0}_{H^2}^2 \Bigg] = C_0 +  C\sum_{k=1}^j \tau \Bigg[\sum_{l=1}^k \tau \norm{\vl}_{H^2}^2\Bigg],
        \end{split}
    \end{equation*}
    where the last equality holds changing the constant $C_0$.
    So, applying the discrete Gronwall inequality stated in \Cref{lemma:discrete_gronwall} leads to
    \begin{equation*}
        \norm{\epsilonvk}_{L^2}^2+ \sum_{k=1}^j\tau \norm{\vk}_{H^2}^2 \leq C.
    \end{equation*}
    Since we already know that $\norm{\vk}_{L^2} \leq C$ thanks to \eqref{eq:energy_inequality}, \eqref{eq:estimate_v} follows.
    Moreover, recalling the trivial identity 
    \begin{equation*}
        \uk = \uu_0 + \sum_{l=1}^k \tau \vl,
    \end{equation*}
    also \eqref{eq:estimate_u} holds true. Finally, by equation \eqref{eq:u_tau}, we can write $\partial_{t}\vt$ as
    \begin{equation*}  
      \begin{split}
            \partial_{t}\vt = \diver\left[h(\zto)\C\epsilonuto\right]-\diver\left[h(\zto)\C\R\phito\right] + \diver\left[a(\zto)\V\epsilonvto\right]
      \end{split}
    \end{equation*}
    and deduce, by comparison, that $\partial_{t}\vt$ is uniformly bounded in $L^2(0,T;L^2)$. Indeed, $h$, $a$, $\C$ and $\V$ are bounded and Lipschitz continuous. The term $\norm{\nabla \zto}_{L^{\infty}(L^p)}$ is uniformly bounded thanks to \eqref{eq:energy_estimate}. Moreover,
    \begin{equation*}
        \norm{\epsilonuto}_{L^{\infty}(H^1)} + \norm{\phito}_{L^{\infty}(H^1)} + \norm{\epsilonvto}_{L^{2}(H^1)}\leq C
    \end{equation*}
    thanks to the estimates \eqref{eq:estimate_u}, \eqref{eq:energy_estimate}, \eqref{eq:estimate_v}, and $H^1 \hookrightarrow L^q$ were $q$ is the H\"older conjugate of $\frac{2p}{p+2}$. So, the estimate \eqref{eq:estimate_li_v} follows. Since $\partial_t \ut = \vto$, \eqref{eq:estimate_u} and \eqref{eq:estimate_v} imply \eqref{eq:estimate_li_u}. \\

    \noindent\textbf{\textit{Higher order estimate for the order parameter.}} Equation \eqref{eq:ch2_tau} can be rewritten as 
    \begin{equation*}
       \psiu'(\phito)-\Delta\phito = \muto - \psin'(\phitu) +h(\ztu)\C(\epsilonutu-\R \phito) : \R - (\phito - \phitu).
    \end{equation*}
    The right-hand side belongs to $L^2(0,T;L^2)$ because $\psin'$ is Lipschitz continuous by Hypothesis \ref{hyp:psi}, $h$ and $\C$ are bounded by Hypotheses \ref{hyp:elastic_viscous_coeff} and \ref{hyp:elastic_viscous_tensors}. More specifically, the following estimate holds:
    \begin{equation*}
        \begin{split}
            \norm{\psiu'(\phito)-\Delta\phito}_{L^2(L^2)} \leq & \, \norm{\muto}_{L^2(L^2)} + C\left( \norm{\phitu}_{L^2(L^2)} + 1 \right)\\
            &+  C\left(\norm{\epsilonutu}_{L^2(L^2)}+ \norm{\phito}_{L^2(L^2)} \right) + \norm{\phito - \phitu}_{L^2(L^2)} \leq C.
        \end{split}
    \end{equation*}
    On the other hand, we have that
    \begin{equation}\label{eq:formal_inequality_estimate_psi}
        \begin{split}
            \norm{\psiu'&(\phito)-\Delta\phito}_{L^2(L^2)}^2\\&= \norm{\psiu'(\phito)}_{L^2(L^2)}^2 + \norm{- \Delta\phito}_{L^2(L^2)}^2 + 2\intto -\Delta \phito \, \psiu'(\phito) \dx \ds\\
            & \geq \norm{\psiu'(\phito)}_{L^2(L^2)}^2 + \norm{-\Delta\phito}_{L^2(L^2)}^2,
        \end{split}
    \end{equation}
    from which estimate \eqref{eq:estimate_psi_derivative} follows. 
    Observe that the inequality in \eqref{eq:formal_inequality_estimate_psi} stands because $\psiu'$ is an increasing continuous function and, therefore, a maximal monotone graph. More explicitly, if we consider its Yosida approximation $\psiu'_{\delta}$ , we have
    \begin{equation*}
        \intto -\Delta \phito \, \psiu_{\delta}'(\phito) \dx \ds = \intto  \psiu_{\delta}''(\phito) |\nabla \phito |^2 \dx \ds \geq 0,
    \end{equation*}
    because $\psiu_{\delta}'$ is monotone and Lipschitz continuous, so $\psiu_{\delta}''$ exists a.e. and it is non-negative. Moreover, $\psiu_{\delta}'(\phito) \to \psiu'(\phito)$ strongly in $L^2(0,T;L^2)$ as $\delta \to 0^+$ (see \cite[][Proposition 2.6, p. 28]{brezis1973}). So, passing to the limit in the previous expression, we deduce what we claimed. Taking into account \eqref{eq:energy_estimate}, we deduce that \eqref{eq:estimate_phi} holds. Notice that the asymmetry between $\phito$ and $\phitu$ in the estimate \eqref{eq:estimate_phi} is a consequence of the fact that $\phitu = \varphi_0$ in $[0,\tau]$ and $\varphi_0$ belongs to $H^1$, not to $H^2$.\\

    \noindent \textbf{\textit{More estimates for the damage.}} From \eqref{eq:z_tau}, we have 
    \begin{equation*}
       -\Delta_p\zto + \beta_{\tau}(\zto) = -\partial_t \zt - \pi(\ztu) - \frac{\hu'(\zto)+\hn'(\ztu)}{2}\C[\epsilonutu-\R\phito] : [\epsilonutu-\R\phito]
    \end{equation*}
    in $L^2(0,T;L^2)$. More specifically, we know that
    \begin{equation*}
        \begin{split}
            \norm{-&\Delta_p\zto + \beta_{\tau}(\zto)}_{L^2(L^2)} \\
            &\leq \,\norm{\partial_t\zt}_{L^2(L^2)} +C\left(\norm{\ztu}_{L^2(L^2)} + 1 +  \norm{\epsilonutu}_{L^4(L^4)}^2 + \norm{\phito}_{L^4(L^4)}^2 \right)\\
            &\leq C\left(\norm{\partial_t\zt}_{L^2(L^2)} + \norm{\ztu}_{L^2(L^2)} + \norm{\utu}_{L^{\infty}(H^2)}^2 + \norm{\phito}_{L^{\infty}(H^1)}^2 + 1\right) \leq C,
        \end{split}
    \end{equation*}
     making use of the previous estimates, the Hypothesis \ref{hyp:pi} according to which $\pi$ is Lipschitz continuous, the fact that $\hu$ and $\hn$ are continuous, the uniform boundedness of $\norm{\zto}_{L^{\infty}(Q)}+\norm{\ztu}_{L^{\infty}(Q)}$ from \eqref{eq:estimate_zk_W1p}, and the embedding $W^{1,p} \hookrightarrow C^0(\Omega)$. 
    On the other hand, 
    \begin{equation*}
        \begin{split}
            \norm{-&\Delta_p\zto + \beta_{\tau}(\zto)}_{L^2(L^2)}^2 \\
            &= \norm{-\Delta_p\zto}_{L^2(L^2)}^2 + \norm{\beta_{\tau}(\zto)}_{L^2(L^2)}^2 + 2\intto -\Delta_p\zto \,\beta_{\tau}(\zto) \dx \ds\\
            &= \norm{-\Delta_p\zto}_{L^2(L^2)}^2 + \norm{\beta_{\tau}(\zto)}_{L^2(L^2)}^2 + 2\intto  \,\beta_{\tau}'(\zto) |\nabla \zto|^p \dx \ds\\
            & \geq \norm{-\Delta_p\zto}_{L^2(L^2)}^2 + \norm{\beta_{\tau}(\zto)}_{L^2(L^2)}^2,
        \end{split}
    \end{equation*}
    where the inequality stands because $\beta_{\tau}'$ is monotone and Lipschitz continuous (so it is a.e. differentiable with positive derivative). Thus, we have proved \eqref{eq:estimate_delta_z} and \eqref{eq:estimate_xi}, employing the fact that $z_0 \in \mathcal{D}(-\Delta_p)$ by Hypothesis \ref{hyp:initial_conditions}. Finally, to conclude the estimate \eqref{eq:estimate_z}, we make use of the inequality stated in \Cref{lemma:regularity_inequality}, from which 
    \begin{equation*}
        \norm{\zto}_{W^{1+\delta,p}} + \norm{\ztu}_{W^{1+\delta,p}} \leq C_{\delta}\left(\norm{-\Delta_p \zto}_{L^2}+\norm{-\Delta_p \ztu}_{L^2} + \norm{\zto}_{L^2} + \norm{\ztu}_{L^2} \right),
    \end{equation*}
    for any $\delta \in (0, 1/p)$.
    Thanks to \eqref{eq:estimate_zk_W1p} that we have already proved, we get \eqref{eq:estimate_z}. Combining \eqref{eq:estimate_z} with the energy estimate \eqref{eq:energy_estimate}, we obtain \eqref{eq:estimate_li_z}.\\
\end{proof}
\subsection{Compactness assertions}
\begin{lemma}\label{lemma:passage_limit}
    There exists a quintuplet $(\varphi,\, \mu,\, \sigma,\, \uu,\, z)$ that satisfy the regularity of \Cref{thm:existence} such that, for a non-relabelled subsequence, we have
    \begingroup
    \allowdisplaybreaks
    \begin{align}
        &\phit \to \varphi &&\text{weakly-}\ast &&\text{in } L^{\infty}(0,T;H^1) \cap H^1(0,T;(H^1)'), \label{eq:limit_li_phi_w}\\
        & &&\text{strongly} &&\text{in } C^0([0,T];L^r) \,\forall r \in [1,6), \label{eq:limit_li_phi_s}\\
        &\phito \to \varphi &&\text{weakly-}\ast &&\text{in } L^{\infty}(0,T;H^1) \cap L^2(0,T;H^2), \label{eq:limit_phi_w}\\
        &\phito,\phitu \to \varphi &&\text{strongly} &&\text{in } L^{r}(0,T;L^r) \, \forall r \in [1,6) \text{ and a.e. in } Q, \label{eq:limit_phi_s}\\
        &\psiu'(\phito) \to \psiu'(\varphi) 
        &&\text{weakly} &&\text{in } L^{2}(0,T;L^2), \label{eq:limit_psiu_w}\\
        &\psin'(\phitu) \to \psin'(\varphi) 
        &&\text{strongly} &&\text{in } L^{2}(0,T;L^2), \label{eq:limit_psin_s}\\
        &\muto,\mutu \to \mu &&\text{weakly} &&\text{in } L^{2}(0,T;H^1), \label{eq:limit_mu_w}\\
        &\sigmat \to \sigma &&\text{weakly-}\ast &&\text{in } L^{\infty}(0,T;L^2) \cap L^{2}(0,T;H^1)\cap H^1(0,T;(H^1)'), \label{eq:limit_li_sigma_w}\\
        & &&\text{strongly} &&\text{in } L^{2}(0,T;L^r) \, \forall r \in [1,6) \text{ and a.e. in } Q,\label{eq:limit_li_sigma_s}\\
        &\sigmato \to \sigma &&\text{weakly-}\ast &&\text{in } L^{\infty}(0,T;L^2) \cap L^{2}(0,T;H^1), \label{eq:limit_sigma_w}\\
        &\ut \to \uu &&\text{weakly-}\ast &&\text{in } W^{1,\infty}(0,T;H^1) \cap H^1(0,T;H^2), \label{eq:limit_li_u_w}\\
        & &&\text{strongly} &&\text{in } C^0([0,T];X) \text{ for all } X \text{ s.t. } H^2\hookrightarrow\hookrightarrow X \hookrightarrow L^2,\label{eq:limit_li_u_s}\\
        &\uto,\utu \to \uu &&\text{weakly-}\ast &&\text{in } L^{\infty}(0,T;H^2), \label{eq:limit_u_w}\\
        & &&\text{strongly} &&\text{in } L^{\infty}(0,T;X) \text{ for all } X \text{ s.t. } H^2 \hookrightarrow\hookrightarrow X \hookrightarrow L^2,\label{eq:limit_u_s}\\
        &\vt \to \partial_t \uu &&\text{weakly-}\ast  &&\text{in }  L^{\infty}(0,T;H^1) \cap H^1(0,T;L^2), \label{eq:limit_li_v_w} \\
        &\vto \to \partial_t \uu &&\text{weakly-}\ast  &&\text{in } L^{\infty}(0,T;H^1) \cap L^{2}(0,T;H^2), \label{eq:limit_v_w}\\
        &\zt \to z &&\text{weakly-}\ast  &&\text{in } L^{\infty}(0,T;W^{1,p}) \cap L^2(0,T;W^{1+\delta,p}) \cap H^1(0,T;L^2), \label{eq:limit_li_z_w}\\
        & && \text{strongly}  &&\text{in } L^s(0,T;W^{1,p})  \, \forall s \in [1,+\infty) \text{ and a.e. in } Q,\label{eq:limit_li_z_s}\\
        &\zto,\ztu \to z &&\text{weakly-}\ast  &&\text{in } L^{\infty}(0,T;W^{1,p}) \cap L^2(0,T;W^{1+\delta,p}) \, \forall \delta \in \left(0,1/p\right),  \label{eq:limit_z_w}\\
        & &&\text{strongly}   &&\text{in } L^{s}(0,T;W^{1,p})  \, \forall s \in [1,+\infty) \text{ and a.e. in } Q, \label{eq:limit_z_s}\\
        &-\Delta_p\zto \to -\Delta_p z &&\text{weakly}  &&\text{in } L^2(0,T;L^2), \label{eq:limit_delta_z_w}\\
        &\beta_{\tau}(\zto) \to \xi &&\text{weakly}  &&\text{in } L^2(0,T;L^2) \text{ with } \xi \in \beta(z). \label{eq:limit_beta_z_w}
    \end{align}
    \endgroup
\end{lemma} 
\begin{proof}
    Most of the convergences are obvious from \Cref{prop:a_priori_estimate} and standard compactness results (Banach–-Alaoglu theorem and  Aubin--Lions theorem); this way, we immediately obtain \eqref{eq:limit_phi_w}, \eqref{eq:limit_li_phi_w}--\eqref{eq:limit_li_phi_s}, \eqref{eq:limit_mu_w}--\eqref{eq:limit_u_w},  \eqref{eq:limit_v_w}--\eqref{eq:limit_z_w}. In the following, we will prove the other ones, focusing on the case $d=3$, which is the most challenging due to weaker embeddings and interpolation inequalities available. The case  $d=2$ can be treated in a similar but easier way. Notice that it is easy to identify the limit of a piecewise constant interpolant and its retarded function. For example, let's prove that $\mutu$ and $\muto$ converge to the same limit. From \eqref{eq:estimate_mu}, we know that $\mutu \to \mu$ and $\muto \to \nu$ weakly in $L^2(0,T;H^1)  \hookrightarrow L^2(0,T;L^2)$. Moreover, we recall that, by definition,
    \begin{equation}\label{eq:traslation_formula_interpolants}
        \muto(t) = \mutu(t+\tau) \qquad \text{for a.e. } t \in (0,T-\tau).
    \end{equation}
    Take a test function $\rho \in C^{\infty}_c(\Omega \times (0,T))$. Since it has compact support, there exists a $\epsilon >0$ such that $\text{supp}(\rho) \subseteq \Omega \times (\epsilon, T-\epsilon)$ and definitively $2\tau < \epsilon$. By a simple change of variables, taking \eqref{eq:traslation_formula_interpolants} into account, we have
    \begin{equation*}
        \begin{split}
            \intto \muto \rho \dx \dt &= \int_{\epsilon}^{T-\epsilon}\!\!\into\mutu(x,t+\tau)\rho(x,t) \dx \dt = \int_{\epsilon+\tau}^{T + \tau -\epsilon}\!\!\into\mutu(x,s)\rho(x,s-\tau) \dx \ds\\
            &=\intto \mutu(x,s)\rho(x,s-\tau) \dx \ds 
            \to \intto \mu  \rho \dx \ds.
        \end{split}
    \end{equation*}
    Here we used the fact that $\rho(\cdot,\cdot-\tau)$ is still a test function with compact support in $(\epsilon+\tau, T+\tau -\epsilon) \subseteq (0,T-\tau)$ and then we passed to the limit because we have the product of a weak convergent sequence and a strong convergent one in $L^2(0,T;L^2)$. On the other hand, we also know that
    \begin{equation*}
        \intto \muto \rho \dx \dt \to \intto \nu \rho \dx \dt.
    \end{equation*}
    Thus, by uniqueness of the limit and the Fundamental Lemma of the Calculus of Variations, we conclude that $\mu = \nu$.
    In the following, we will discuss the less immediate limits of the statement.  To prove \eqref{eq:limit_phi_s}, we initially show that $\phito,\phitu \to \varphi$ strongly in $L^2(0,T;L^2)$. Rewriting the piecewise linear interpolant $\phit$ as
    \begin{equation*}
        \phit(t)=\phitu(t) + \frac{t-t_{k-1}}{\tau} \left[\phito(t)-\phitu(t)\right]
    \end{equation*}
    for every $t \in I_{\tau}^k$, then
    \begin{equation*}
        \begin{split}
            \norm{\phit-\phitu}^2_{L^2(L^2)} &= \sum_{k=1}^{K_{\tau}}\intk \into \left(\frac{t-t_{k-1}}{\tau}\right)^2(\phito(t)-\phitu(t))^2 \dx\dt\\ &\leq \norm{\phito-\phitu}_{L^2(L^2)}^2 \leq \tau C \to 0,
        \end{split}
    \end{equation*}
    where the last inequality is due to \eqref{eq:estimate_phi_jump}. On the other hand, by \eqref{eq:limit_li_phi_s}, $\phit$ goes to $\varphi$ strongly in $L^2(0,T;L^2)$, so also $\phitu \to \varphi$ in $L^2(0,T;L^2)$ and, using again \eqref{eq:estimate_phi_jump}, the same stands for $\phito$. We can also deduce that, along a non-relabelled subsequence, $\phito,\phitu \to \varphi$ a.e. in $Q$. 
    Since $\norm{\phitu}_{L^6(L^6)}, \norm{\phito}_{L^6(L^6)} \leq C$, as ensured by \eqref{eq:estimate_phi} and by the embedding $L^{\infty}(0,T;H^1) \hookrightarrow L^{\infty}(0,T;L^6)$, and given that $\phito, \phitu \to \varphi$ pointwise a.e., it follows that $\phito,\phitu \to \varphi$ in $L^r(0,T;L^r)$ for every $r \in [1,6)$, so \eqref{eq:limit_phi_s} stands. From \eqref{eq:estimate_psi_derivative}, $\norm{\psiu'(\phito)}_{L^2(L^2)} \leq C$; moreover, $\psiu$ is continuous and $\phito \to \varphi$ a.e., so $\psiu'(\phito) \to \psiu'(\varphi)$ a.e. in $Q$. Hence, we have also \eqref{eq:limit_psiu_w}. By \eqref{eq:limit_phi_s} and the Lipschitz continuity of $\psin'$, we get \eqref{eq:limit_psin_s}.
    In order to prove \eqref{eq:limit_u_s}, we start by noticing that for every $t \in I_{\tau}^k$
    \begin{equation*}
        \ut(t) = \utu(t) + \frac{t-t_{k-1}}{\tau}\left[\uto(t)-\utu(t)\right] =  \utu(t) + \frac{t-t_{k-1}}{\tau} \int_{I_{\tau}^k} \vto \ds,
    \end{equation*}
    from which, using \eqref{eq:estimate_v} and $H^2 \hookrightarrow X$, it follows that 
    \begin{equation*}
        \norm{\ut -\utu}_{X} \leq  \norm{\vto}_{L^2(X)} \tau^{1/2}\leq C \tau^{1/2} \to 0.
    \end{equation*}
    Since we already know that $ \ut \to \uu$ strongly in $L^{\infty}(0,T;X)$ by \eqref{eq:limit_li_u_s}, this inequality leads to \eqref{eq:limit_u_s}. Finally, we prove the convergences regarding the damage. From Aubin--Lions compactness result, $L^2(0,T;W^{1+\delta,p}) \cap H^1(0,T;L^2) \hookrightarrow \hookrightarrow L^2(0,T;W^{1,p})$ so, using \eqref{eq:estimate_li_z} and \eqref{eq:estimate_z}, along a subsequence $\zt \to z$ strongly in $L^2(0,T;W^{1,p})$. Since $\zt$ is bounded in $L^{\infty}(0,T;W^{1,p})$, we obtain \eqref{eq:limit_li_z_s}.
    As we have already observed before, for every $t \in I_{\tau}^k$ it holds
    \begin{equation*}
        \zt(t) = \zto(t) - \frac{t_{k}-t}{\tau} \int_{I_{\tau}^k} \partial_t \zt \ds
    \end{equation*}
    and, as a consequence,
    \begin{equation*}
        \norm{\zt -\zto}_{L^{\infty}(L^2)} \leq \norm{\partial_t\zt}_{L^2(L^2)} \tau^{1/2}  \leq C \tau^{1/2} \to 0.
    \end{equation*}
    Hence, we deduce that, along a subsequence, $\zt -\zto \to 0$ a.e. in $Q$. Since we know that
    \begin{equation*}
        \norm{\zt - \zto}_{L^{\infty}(Q)}  \leq C \norm{\zt}_{L^{\infty}(W^{1,p})} +  \norm{\zto}_{L^{\infty}(W^{1,p})} \leq C, 
    \end{equation*}
    we obtain that $\zt - \zto \to 0 $ strongly in $L^s(0,T;L^s)$ for every $s \in [0,+\infty)$. It trivially follows that $\zt - \zto \to 0 $ strongly in $L^s(0,T;L^t)$ for every $s,\,t \in [0,+\infty)$. Now we want to prove that a subsequence converges strongly in $L^2(0,T;W^{1,p})$. To reach our purpose, we employ the following inequality of Gagliardo--Nirenberg type for fractional Sobolev spaces (see \cite[][Theorem 1, p. 1356]{Brezis_Mironescu_2018} for further details)
    \begin{equation*}
        \norm{\zt -\zto}_{W^{1,p}} \leq C \norm{\zt -\zto}_{L^p}^{\theta} \norm{\zt -\zto}_{W^{1+\delta,p}}^{1-\theta} 
    \end{equation*}
    with $\theta = \delta/(1+\delta)$. Taking the square of this inequality, integrating over the time interval $(0,T)$ and using the H\"older inequality leads to
    \begin{equation*}
       \begin{split}
            \intt& \norm{\zt - \zto}_{W^{1,p}}^2 \dt \leq C \intt \norm{\zt -\zto}_{L^p}^{2\theta} \norm{\zt -\zto}_{W^{1+\delta,p}}^{2(1-\theta)} \dt\\
            &\leq C \left[\intt \norm{\zt -\zto}_{L^p}^{2\theta q} \dt\right]^{1/q} \left[\intt \norm{\zt -\zto}_{W^{1+\delta,p}}^{2(1-\theta)q'} \dt \right]^{1/q'}
       \end{split} 
    \end{equation*}
    where $q= 1/\theta = (1+\delta)/ \delta$ and $q'=q/(1-q)=1/(1-\theta)$ (so that $2\theta q' = 2(1-\theta)q=2$). Hence, we have
    \begin{equation*}
        \norm{\zt-\zto}_{L^2(W^{1,p})} \leq C \norm{\zt-\zto}_{L^2(L^p)}^{2/q} \norm{\zt-\zto}_{L^2(W^{1+\delta,p})}^{2/q'} \leq C  \norm{\zt-\zto}_{L^2(L^p)}^{2/q} \to 0.
    \end{equation*}
    This strong convergence, combined with the boundedness of $\zt -\zto$ in $L^{\infty}(0,T;W^{1,p})$ (that we have from \eqref{eq:estimate_z}), gives us $\norm{\zt-\zto}_{L^s(W^{1,p})} \to 0$ for every $s \in [0,+\infty)$. Since we already know \eqref{eq:limit_li_z_s}, we have \eqref{eq:limit_z_s}.
    Because of \eqref{eq:estimate_delta_z}, it exists a $w \in L^2(0,T;L^2)$ such that, along a non-relabeled subsequence, $-\Delta_p\zto \rightharpoonup w$ in $L^2(0,T;L^2)$. Then, recalling that $\zto \to z$ strongly in $L^2(0,T;L^2)$ and that the operator $-\Delta_p : L^2 \to L^2$ is maximal monotone so it is strong-weak closed (see \cite[][Proposition 2.5, p. 27]{brezis1973}), we may identify $w = -\Delta_p z$, which proves \eqref{eq:limit_delta_z_w}. Finally, from \eqref{eq:estimate_xi}, we deduce that it exits a $\xi \in L^2(0,T;L^2)$ such that $\beta_{\tau}(\zto) \rightharpoonup \xi$ in $L^2(0,T;L^2)$. Since $\beta$ is maximal monotone, $\beta_{\tau}$ is its Yosida approximation and $\zto \to z$ strongly in $L^2(0,T;L^2)$, using \cite[][Proposition 1.1, p. 42]{barbu1976}, we deduce that $\xi \in \beta(z)$ so \eqref{eq:limit_beta_z_w} stands.\\
\end{proof}

\subsection{Passage to the limit in the discrete system} \label{subsection:passage_to_the_limit}
Now we have all the instruments necessary to prove our main result, \Cref{thm:existence}. We want to exploit the compactness result \Cref{lemma:passage_limit}, proving that the limit we found is a weak solution to our problem in the sense of \Cref{defn:weak_solution}.\\ 

\noindent \textit{\textbf{Cahn--Hilliard equation.}} 
In \eqref{eq:ch1_tau}, we can easily pass to the weak limit in the terms on the left-hand side and in the second term on the right-hand side using convergence \eqref{eq:limit_li_phi_w} and \eqref{eq:limit_mu_w}. Given a $\zeta \in L^2(0,T;H^1)$, we want to prove that
\begin{equation*}
\begin{split}
    \intto \bigg(\frac{\lambda_p \sigmato}{1+|W_{,\varepsilon}(\phitu,\epsilonutu,\ztu)|}&-\lambda_a + \fto\bigg)g(\phitu,\ztu)\zeta \dt \dx\\
    &\to \intto \bigg(\frac{\lambda_p \sigma}{1+|W_{,\varepsilon}(\varphi,\varepsilon(\uu),z)|}-\lambda_a + f\bigg)g(\varphi,z) \zeta \dt \dx.
\end{split}
\end{equation*}
 First, we note that $\phitu$ (resp. $\ztu$, $\epsilonutu$) converges to $\varphi$ (resp. $z$, $\varepsilon(\uu))$ a.e. in $Q$ because of \eqref{eq:limit_phi_s} (resp.  \eqref{eq:limit_z_s}, \eqref{eq:limit_u_s}). Since $g$ and $W_{,\varepsilon}$ are continuous, $g(\phitu,\ztu) \to g(\varphi,z)$ and $W_{,\varepsilon}(\phitu,\epsilonutu,\ztu) \to W_{,\varepsilon}(\varphi,\varepsilon(\uu),z)$ a.e. in $Q$.  Moreover, $g$ is bounded and $1 + |W_{,\varepsilon}| \geq 
 1$. Finally, we recall that $\sigmato \rightharpoonup
 \sigma$ weakly in $L^2(0,T;L^2)$ from \eqref{eq:limit_sigma_w}, and $\fto \to f$ strongly in $L^2(0,T;L^2)$, so the above convergence holds.
  In \eqref{eq:ch2_tau}, exploiting convergences \eqref{eq:limit_mu_w}, \eqref{eq:limit_phi_w}, \eqref{eq:limit_psiu_w} and \eqref{eq:limit_psin_s} we can immediately pass to the weak limit in all the terms except in $W_{,\varphi}(\phito,\epsilonutu,\ztu)$. However, for every $\rho \in L^2(0,T;L^2)$, we have  
\begin{equation*}
    -\intto h(\ztu)(\epsilonutu-\R\phito):\C\R \rho \dx\dt \to -\intto h(z)(\epsilonu-\R\varphi):\C\R \rho \dx\dt, 
\end{equation*}
because $\epsilonutu -\R\phito \rightharpoonup \epsilonu-\R\varphi$ weakly in $L^2(0,T;L^2)$ (from \eqref{eq:limit_u_w} and \eqref{eq:limit_phi_w}) and $\C\R h(\ztu)\rho \to \C\R h(z)\rho$ strongly in $L^2(0,T;L^2)$. This last convergence holds true since $\C$ is bounded, $h$ is continuous and bounded, $\ztu \to z$ a.e. in $Q$ from \eqref{eq:limit_z_s}, so we can apply the Dominated Convergence Theorem.\\  
    
\noindent \textit{\textbf{Nutrient equation.}} Rewriting explicitly \eqref{eq:sigma_tau}, it holds
\begin{equation*}
    \begin{split}
        \intt &\duality{\partial_t \sigmat}{\zeta}_{H^1} \dt + \intto \nabla\sigmato \cdot \nabla\zeta \dx\dt + \alpha \intTg (\sigmato-\sigmagto)\zeta\dA\dt \\
        &= \intto \left[-\lambda_c\sigmato g(\phitu,\ztu) + \Lambdac(\ztu)(\sigmacto-\sigmato)\right]\zeta\dx\dt
    \end{split}
\end{equation*}
for every $\zeta \in L^2(0,T;H^1)$. 
As we have already pointed out, $g(\phitu,\ztu)\zeta \to g(\varphi,z)\zeta$ strongly in $L^2(0,T;L^2)$ and in the same way one can prove that $\Lambdac(\ztu)\zeta \to \Lambdac(z)\zeta$ strongly in $L^2(0,T;L^2)$. So, 
\begin{equation*}
   \begin{split}
       \intto \left[-\lambda_c\sigmato g(\phitu,\ztu) + \Lambdac(\ztu)(\sigmacto-\sigmato)\right]&\zeta\dx\dt\\ 
       \to &\intto \left[-\lambda_c\sigma g(\varphi,z) + \Lambdac(z)(\sigma_c-\sigma)\right]\zeta\dx\dt
   \end{split} 
\end{equation*}
because we also know that $\sigmato \rightharpoonup \sigma$ weakly in $L^2(0,T;L^2)$ thanks to \eqref{eq:limit_sigma_w}. Regarding the term with the boundary integral, we recall that the trace operator $H^1 \to L^2_{\Gamma}$ is linear and continuous. Thus, the weak convergence $\sigmato-\sigmagto \rightharpoonup \sigma-\sigma_{\Gamma}$ in $L^2(0,T;H^1)$, that we have from \eqref{eq:limit_sigma_w} and by construction of $\sigmagto$, leads to the weak convergences of the traces in $L^2(0,T;L^2_{\Gamma})$. All the other terms converge using \eqref{eq:limit_li_sigma_w} and \eqref{eq:limit_sigma_w}. Finally, $0 \leq \sigma \leq M$ because $\sigmat$ satisfies this property and, thanks to \eqref{eq:limit_li_sigma_s}, we have pointwise
convergence a.e. in $Q$.\\
    
\noindent \textit{\textbf{Displacement equation.}} For every $\boldsymbol{\rho} \in L^2(0,T;L^2)$, the following equality holds:
\begin{equation*}
  \begin{split}
        &\intto \partial_t \vt \cdot \boldsymbol{\rho} \dx \dt 
        - \intto h'(\zto)\C\left[ \epsilonuto -\R\phito\right] \nabla\zto \cdot \boldsymbol{\rho} \dx \dt\\
        & \quad - \intto h(\zto)\diver\left(\C\epsilonuto - \C\R\phito \right) \cdot \boldsymbol{\rho}\dx \dt
        - \intto a'(\zto)\V\epsilonvto\nabla\zto \cdot \boldsymbol{\rho} \dx \dt\\ 
        & \quad - \intto a(\zto)\diver\left(\V\epsilonvto\right) \cdot \boldsymbol{\rho} \dx \dt = 0.
  \end{split}
\end{equation*}
Thanks to \eqref{eq:limit_li_v_w}, we can pass to the weak limit in the first term. For the other more complicated addends, we proceed explicitly. Regarding the second term, we want to prove that
\begin{equation*}
    \intto  h'(\zto)\C\left[ \epsilonuto -\R\phito\right] \nabla\zto \cdot\boldsymbol{\rho} \dx \dt \to \intto  h'(z)\C\left[ \varepsilon(\uu) -\R\varphi\right] \nabla z \cdot \boldsymbol{\rho} \dx \dt. 
\end{equation*}
As we have already exploited, $\zto \to z$ a.e. in $Q$ and $h'$ is continuous, so $h'(\zto) \to h'(z)$ a.e. in Q. Moreover, from \eqref{eq:estimate_z} we know that $\norm{\zto}_{L^{\infty}(Q)} \leq C$ so, since $h'$ is continuous, $\norm{h'(\zto)}_{L^{\infty}(Q)} \leq C$. From \eqref{eq:limit_z_s}, choosing $s=p$, we get that $\nabla \zto \to \nabla z$ in $L^p(0,T;L^p)$. Hence, $h'(\zto)\nabla\zto \cdot \boldsymbol{\rho} \to h'(z)\nabla z \cdot \boldsymbol{\rho}$ strongly in $L^q(0,T;L^q)$ with $q=\frac{2p}{p+2}$. Let $q'$ be the H\"older conjugate of $q$, then it is easy to verify that $q' \in [2,6)$ if $d=3$ and $q'\in [2,+\infty)$ if $d=2$. From the boundedness of $\C$, \eqref{eq:limit_phi_w} and \eqref{eq:limit_u_w}, we have that $\C\left[ \epsilonuto -\R\phito\right] \rightharpoonup \C\left[ \varepsilon(\uu) -\R\varphi\right]$ weakly in $L^{q'}(0,T;L^{q'})$, so the desired convergence follows. To prove that
\begin{equation*}
    \intto h(\zto)\diver\left(\C\epsilonuto - \C\R\phito \right) \cdot\boldsymbol{\rho} \dx \dt \to \intto h(z)\diver\left(\C\epsilonu - \C\R\varphi \right) \cdot \boldsymbol{\rho} \dx \dt, 
\end{equation*}
we observe that $h$ is continuous and bounded and $\zto \to z$ a.e. so, thanks to the Dominated Convergence Theorem, $h(\zto)\boldsymbol{\rho} \to h(z)\boldsymbol{\rho}$ in $L^2(0,T;L^2)$. Moreover, by \eqref{eq:limit_u_w}, \eqref{eq:limit_phi_w}, and since $\C$ is bounded and Lipschitz, $\diver\left(\C\epsilonuto - \C\R\phito \right) \rightharpoonup \diver\left(\C\varepsilon(\uu) - \C\R\varphi \right)$ weakly in $L^2(0,T;L^2)$. 
Now we take into consideration the fourth term, and we are going to show that
\begin{equation*}
    \intto a'(\zto)\V\epsilonvto\nabla\zto \cdot \boldsymbol{\rho}   \dx \dt \to \intto a'(z)\V\epsilonv\nabla z \cdot \boldsymbol{\rho} \dx \dt.
\end{equation*}
Since $a'$ is continuous and $\zto \to z$ a.e. by \eqref{eq:limit_z_s}, $a'(\zto) \to a'(z)$ a.e. in $Q$. Exploiting $a'$ boundedness (that follows from the fact that $a$ is Lipschitz), by the Dominated Convergence Theorem $a'(\zto)\boldsymbol{\rho} \to a'(z) \boldsymbol{\rho}$ strongly in $L^2(0,T;L^2)$.
 Moreover, by \eqref{eq:limit_v_w} and \eqref{eq:limit_phi_w}, $\epsilonvto  \overset{\ast}{\rightharpoonup} \varepsilon(\partial_t \uu) $ weakly-$\ast$ in $L^{\infty}(0,T;L^2) \cap L^2(0,T;H^1) \hookrightarrow L^{2p/d}(0,T;L^{2p/(p-2)})$, where the embedding holds true because of Gagliardo--Nirenberg's inequality. More precisely, we apply   \Cref{thm:Gagliardo-Nieremberg_ineq} with
 \begin{equation*}
     r \in \left(2, \frac{2d}{d-2}\right), \qquad q=2, \qquad s=1, \qquad \alpha = d \left(\frac{1}{r}-\frac{d-2}{2d}\right).
 \end{equation*}
  Finally, from \eqref{eq:limit_z_s} with $s=2p/(p-d)$, we get $\nabla \zto \to \nabla z$ strongly in $L^{2p/(p-d)}(0,T;L^p)$. So, we have concluded, because
\begin{equation*}
    \frac{1}{2}+\frac{1}{2p/d}+\frac{1}{2p/(p-d)}=1,\qquad \frac{1}{2}+\frac{1}{2p/(p-2)} + \frac{1}{p} = 1.
\end{equation*}
Lastly, we aim to show that
\begin{equation*}
    \intto a(\zto)\diver\left(\V\epsilonvto \right) \cdot\boldsymbol{\rho} \dx \dt \to \intto a(z)\diver\left(\V\epsilonv \right) \cdot \boldsymbol{\rho} \dx \dt. 
\end{equation*}
Continuity and boundedness of $a$, convergences a.e. of $\zto$ from  \eqref{eq:limit_z_s}, and  the Dominated Convergence Theorem lead to $a(\zto)\boldsymbol{\rho}\to a(z)\boldsymbol{\rho}$ strongly in $L^2(0,T;L^2)$. From \eqref{eq:limit_v_w} we have that $\vto \rightharpoonup \partial_t \uu$ weakly in $L^2(0,T;H^2)$ and from Hypothesis \ref{hyp:elastic_viscous_coeff} $\V$ is bounded and Lipschitz continuous. Thus, the last term of the displacement equation passes to the limit.\\
    
\noindent \textit{\textbf{Damage equation.}} 
We discuss only the less immediate term. Consider a test function $\rho \in L^2(0,T;L^2)$. We will prove that
\begin{equation*}
    \begin{split}
        \intto\frac{\hu'(\zto)+\hn'(\ztu)}{2}(\epsilonutu-\R\phito):\C&(\epsilonutu-\R\phito) \rho \dx\dt\\
        \to &\intto \frac{h'(z)}{2}(\epsilonu-\R\varphi):\C(\epsilonu-\R\varphi) \rho \dx\dt. 
    \end{split}
\end{equation*}
Since $\zto$, $\ztu$ $\to z$ a.e. in $\Omega$, $\hu'$, $\hn'$ are continuous, and $\hu' + \hn' = h'$, we have $\frac{\hu'(\zto)+\hn'(\ztu)}{2}\to \frac{h'(z)}{2}$ a.e. in $\Omega$. Moreover, $\norm{\zto}_{L^{\infty}(Q)}$, $\norm{\ztu}_{L^{\infty}(Q)}$ are uniformly bounded by \eqref{eq:estimate_z}. It follows that $\norm{\hu'(\zto)}_{L^{\infty}(Q)}$, $\norm{\hu'(\ztu)}_{L^{\infty}(Q)} \leq C$. Using the Dominated Convergence Theorem, we deduce that $\frac{\hu'(\zto)+\hn'(\ztu)}{2} \rho\to \frac{h'(z)}{2}\rho$ strongly in $L^2(0,T;L^2)$. From \eqref{eq:limit_u_s}, choosing $X=W^{1,4}$, and  from \eqref{eq:limit_phi_s}, choosing $r=4$, we get that $\epsilonutu - \R\phito \to \epsilonu - \R\varphi$ strongly in $L^4(0,T;L^4)$. Since $\C$ is bounded, we have the desired convergence.  

\section*{Acknowledgments}
The author wishes to express her gratitude to Professor Pierluigi Colli and Professor Elisabetta Rocca for introducing her to this research area and for their numerous insightful suggestions, without which this work would not have been possible. 
The author also wishes to thank the anonymous referees for their careful reading and constructive comments, which have contributed to improving the manuscript. 
The author is a member of  GNAMPA (Gruppo Nazionale per l'Analisi Matematica, la Probabilit\`a e le loro Applicazioni) of INdAM (Istituto Nazionale di Alta Matematica).
This research activity has been supported by the MIUR-PRIN Grant 2020F3NCPX ``Mathematics for industry 4.0 (Math4I4)''.

\printbibliography

\end{document}